\documentclass[11pt,a4paper,reqno]{amsart}

\usepackage{a4wide}
\usepackage[utf8x]{inputenc}
\usepackage[T1]{fontenc}
\usepackage{amsthm,amsrefs}
\usepackage{thmtools}
\usepackage{nccmath}

\usepackage{newtxtext}
\usepackage{newtxmath}
\usepackage[cal=euler,bb=ams,frak=euler,scr=boondoxo]{mathalfa}

\usepackage[all]{xy}
\usepackage{xcolor}
\usepackage{enumitem}
\setenumerate[1]{itemsep=0pt,partopsep=0pt,parsep=\parskip,topsep=5pt}
\setitemize[1]{itemsep=0pt,partopsep=0pt,parsep=\parskip,topsep=5pt}
\setdescription{itemsep=0pt,partopsep=0pt,parsep=\parskip,topsep=5pt}

\usepackage{graphicx}
\usepackage{hyperref}
\hypersetup{
hyperindex,
plainpages=false,
bookmarks=true,
bookmarksopen=true,
bookmarksnumbered,
colorlinks=true,
linkcolor=purple,          
citecolor=blue,        
filecolor=magenta,       
urlcolor=violet,          
}

\theoremstyle{plain}
\newtheorem{theorem}{Theorem}[section]
\newtheorem{corollary}[theorem]{Corollary}
\newtheorem{lemma}[theorem]{Lemma}
\newtheorem{proposition}[theorem]{Proposition}
\newtheorem{remark}[theorem]{Remark}

\theoremstyle{definition}
\newtheorem{definition}[theorem]{Definition}
\theoremstyle{remark}

\newtheorem*{theorem*}{Theorem \ref{mainthm}}

\numberwithin{equation}{section}

\newcommand{\wc}{\rightharpoonup} 
\newcommand{\grass}[2]{\mathbf{G}(#1,#2)}
\newcommand{\Var}{\mathbf{V}}     
\newcommand{\var}{\mathbf{v}}     
\newcommand{\oball}{\mathbf{U}}
\newcommand{\cball}{\mathbf{B}}
\newcommand{\HM}{\mathcal{H}}
\newcommand{\ud}{\ensuremath{\,\mathrm{d}}}

\DeclareMathOperator{\VarTan}{VarTan}   
\DeclareMathOperator{\spt}{spt}
\DeclareMathOperator{\id}{id}
\DeclareMathOperator{\Tan}{Tan}
\DeclareMathOperator{\Lip}{Lip}
\DeclareMathOperator{\dist}{dist}
\DeclareMathOperator{\diam}{diam}
\DeclareMathOperator{\reach}{reach}
\DeclareMathOperator{\im}{im}%
\DeclareMathOperator{\Unp}{Unp}%
\DeclareMathOperator{\ap}{ap} 
\DeclareMathOperator{\intr}{int} 
\newcommand{\scale}[1]{\boldsymbol{\mu}_{#1}}

\newcommand{\mr}{\mathop{\vrule height 1.6ex depth 0pt width
0.13ex\vrule height 0.13ex depth 0pt width 1.3ex}\nolimits}
\newcommand{\cech}{\check{\mathsf{H}}}

\newcommand{\QM}{\mathsf{QM}}
\newcommand{\AM}{\mathsf{AM}}
\newcommand{\cf}{\mathscr{C}}
\newcommand{\ann}{\mathbf{A}}
\newcommand{\dmn}{\mathrm{dmn}}

\newcommand{\Rnum}[1]{\uppercase\expandafter{\romannumeral #1}}
\newcommand{\ora}[1]{{\scriptstyle\xrightarrow{#1}}}

\newcommand{\medcup}{\medop{\bigcup}}

\begin{document}
\title[Quasiminimal sets and Plateau's problem]{Quasiminimal sets and 
Plateau's problem with {\v C}ech homological conditions on $C^2$ submanifold}
\author{Yangqin Fang}
\address{\parbox{\linewidth}{\strut Yangqin Fang\\ School of Mathematics and Statistics\\
Huazhong University of Science and Technology\\
430074, Wuhan, P.R. China\strut}}
\email{yangqinfang@hust.edu.cn}
\begin{abstract}
	Let $\Omega\subseteq \mathbb{R}^n$ be an $m$-dimensional closed submanifold
	of class $C^2$, $d$ be a positive integer between 1 and $m$. We will study
	the geometric and topological proprieties of quasiminimal sets in $\Omega$,
	and show that a minimizing sequence of $d$-sets converges to a minimal set 
	in the sense of weak topology. Following from that, we can solve the Plateau's
	problem of dimension $d$ on $\Omega$ with \v{C}ech homological conditions.
\end{abstract}
\maketitle

\section{Introduction}
Plateau's problem raised from the 18th century, has flourished since the middle 20th
century. The first solution to the Plateau's problem in the classical form
came in the early 1930's by Douglas \cite{Douglas:1931} and Rad\'o 
\cite{Rado:1930} independently. In the 1960's, Federer and Fleming  \cite{FF:1960}
introduced currents, and solved the Plateau's problem as a mass minimization
problem of currents. There are simultaneously raised the size minimization
problem of currents. However, the latter is still open. Contemporaneously,
Riefenberg \cite{Reifenberg:1960} considered the Plateau's  problem involving
homological boundary conditions, and he demonstrated the  existence of solutions
in euclidean spaces when the coefficient group is  compactly abelian and when
the support of the algebraic boundary is compact  and one-dimension lower.

Decades have passed since Riefenberg proposed Plateau's problem with {\v C}ech
homological conditions, which was developed by many authors in the last few 
years, and there are indeed plenty of existence results in the standard
euclidean spaces, but little is known when we place it in a more complicated
ambient space, such as euclidean space with holes, manifold, Hilbert or Banach
space etc. There are several difficulties that obstruct its development. Indeed,
when we consider a minimizing sequence of sets for the Plateau's problem, if 
it converges in Hausdorff distance to a set, then the limit set is actually a
competitor, but the problem is that we do not know its Hausdorff measure on
account of lack of  lower semi-continuity for the Hausdorff measure in general, so
we do not know if the limit set is a minimizer or not. We can also take a
minimizing sequence of sets, then consider a sequence of Radon measures which
is just the restriction of the Hausdorff measure on the minimizing sequence, and
suppose that it converges in weak topology to a Radon measure, we automatically
get the lower semi-continuity of the total mass, but it is hard to get the 
density one for the limit measure. 

In this paper, we aim to solve the Plateau's problem with \v{C}ech homological
conditions on a closed submanifold in $\mathbb{R}^{n}$ of class $C^2$, but
what we get is much more than that the Plateau's problem has just been solved,
indeed a lot of properties for quasiminimal sets will be revealed.
We first investigate some geometric and topological properties
of $\QM_d(\Omega,U,M,\varepsilon)$ the class of quasiminimal sets, see
Definition \ref{def:quasimimimal} for the details, and which is indeed
contains all of the generalized quasiminimal sets defined by David in
\cite{David:2009} and the $(\mathbf{M},\varepsilon,\delta)$-minimal sets defined
by Almgren. For any  open set $U\subseteq \mathbb{R}^n$, considering
$\QM_d(U,U,M,\varepsilon)$, by a deformation theorem, for example
\ref{thm:ffp}, it is not difficult to get  that all such quasiminimal sets are
$d$-rectifiable if $\varepsilon(r)\leq  cr^d$ for some constant $c>0$ small
enough, see Corollary \ref{co:rect} for  the details, which is indeed a
simpler proof of rectifiablity of quasiminimal sets. Furthermore, if
$\int_0^{r_0}r^{-d-1}\varepsilon(r)\ud r<\infty$ for some $r_0>0$, then
quasiminimal sets in $\QM_d(U,U,M,\varepsilon)$ are local Ahlfors-regular, see
Corollary \ref{co:AR}. Let's consider a sequence of quasiminimal sets which
converges to a Radon measure in the sense of weak topology. Actually the
support of the limit measure is also quasiminimal, and we have a lower
semi-continuity property and an upper semi-continuity property  for the
Hausdorff measure.
\begin{theorem} \label{thm:QM}
	Suppose $1\leq d\leq m\leq n$. Let $\Omega\subseteq \mathbb{R}^n$ be a
	closed $m$-dimensional submanifold of class $C^2$, $\partial \Omega$ be its
	boundary. Let $U\subseteq \mathbb{R}^n$ be an open set such that $U\cap
	\Omega\neq \emptyset$ and $U\cap \partial \Omega=\emptyset$. Suppose that
	$\{E_k\}\subseteq \QM_d(\Omega,U,M_k,\varepsilon_k)$, $M(r)=\limsup_{k\to
	\infty}M_k(r)$, $\varepsilon(r)=\limsup_{k\to \infty}\varepsilon_k(r)$,
	$\HM^d\mr (E_k\cap U)\wc \mu$, $E=U\cap \spt\mu$. If $M(0+)<\infty$ and
	$\lim_{r\to 0+} r^{-d}\varepsilon(r)=0$, then 
	\begin{itemize}
		\item $E\in \QM_d(\Omega,U,M,\varepsilon)$;
		\item for any open set $O\subseteq U$, we have that   $ \HM^d(E\cap O)
			\leq \liminf_{k\to \infty}\HM^d(E\cap O)$;
		\item for any compact set $H\subseteq U$, we have that   $ \limsup_{k\to
			\infty}\HM^d(E\cap H)\leq  M(0+)\HM^d(E\cap H)$.
	\end{itemize}
\end{theorem}

This theorem is a generalized form of Theorem 3.4 and Theorem 4.1 in
\cite{David:2003}. Following from the lower semi-continuity of Hausdorff
measure in above theorem and the continuity of \v{C}ech homology with respect
to the inverse limit, we can solve the Plateau's problem with \v{C}ech homological
conditions:

\begin{theorem}\label{thm:RPM}
	Suppose $1\leq d\leq m\leq n$. Let $\Omega\subseteq \mathbb{R}^n$ be an
	$m$-dimensional closed submanifold  of class $C^2$, $\partial \Omega$ be its
	boundary, $B_0\subseteq \Omega$ be a compact subset. Let $G$ be an abelian
	group. Suppose that $\partial \Omega \setminus B_0=\emptyset$, $B_0\neq
	\Omega$, and $L\subseteq \cech_{d-1}(B_0;G)$ is a subgroup. If there exists a
	uniformly  bounded minimizing sequence which spans $L$, then there
	exit	minimizers. In particular, if $\Omega$ is a compact submanifold of
	class $C^2$, or $\mathbb{R}^n\setminus \Omega$ is a compact set, and in
	addition there exist competitors, then there exist minimizers. 
\end{theorem}
We refer the readers to Section 6 for the precise definition of terms
minimizing sequence, minimizers, competitors and spanning.
Let us mention an important consequence of the theorem.
\begin{corollary}
	Let $d$, $m$, $n$, $\Omega$, $B_0$, $G$ and $L$ be the same as in Theorem 
	\ref{thm:RPM}. If $\Omega$ is compact and $\cech_{d-1}(\Omega;G)=0$, then
	there  exit	minimizers for the Plateau's  problem with \v{C}ech homological
	conditions. In particular, it holds for $\Omega=\mathbb{R}^m$ and 
	$\Omega=\mathbb{S}^{m}$.
\end{corollary}

For any positive integer $n\geq 2$ and positive integer $d$ between $1$ and
$n$, and any $d$-plane $P$ in $\mathbb{R}^{n}$, we denote by $P_{\natural}$
the orthogonal projection from $\mathbb{R}^{n}$ onto $P$, by abuse notation,
we also denote it by $P$. We  denote by $\grass{n}{d}$ the grassmannian
manifold which consists of  $d$-planes in $\mathbb{R}^{n}$ through 0 equipped
with metric $\rho(T,P)=\|T_{\natural}-P_{\natural}\|$.  Radon measures on
$\Omega\times \grass{n}{d}$ are called varifolds, we denote by
$\Var_d(\Omega)$ the set of varifolds on $\Omega$, see Definition 3.1 in
\cite{Allard:1972}. For any $V\in \Var_d(\Omega)$, we denote by $\|V\|$ the
Radon measure on $\Omega$ defined by $\|V\|(E)=V(E\times \grass{n}{d})$. Let
$\mathscr{M}(\Omega)$ be the collection of $\HM^d$-measurable sets in $\Omega$.
Let $\var:\{E\in \mathscr{M}(\Omega):\HM^d(E)<\infty\}\to \Var_d(\Omega)$ be a
mapping such that for any $\HM^d$-measurable $d$-rectifiable $E$ and $\varphi
\in C_c(\Omega\times\grass{n}{d},\mathbb{R})$,
$	\var(E)(\varphi)=\int_{E}\varphi(x,\Tan^dE,x))\HM^d(x) $,
and that there is a positive constant $\alpha>0$ satisfying that for any 
$\HM^d$-measurable purely $d$-unrectifiable set $E$, $ \|\var(E)\|\leq \alpha
\HM^d\mr E $. We will study the quasiminimal sets further. Indeed, as a direct
consequence of Theorem \ref{thm:QM}, we have the following result:
\begin{theorem}\label{thm:qmv}
	Let $d$, $m$, $n$, $\Omega$, $U$, $\{M_k\}$, $\varepsilon_k$, $\{E_k\}$,
	$E$, $M$ and $\varepsilon$ be the same as in Theorem \ref{thm:QM}. If
	$M(0+)=1$ and $\lim_{r\to 0+} r^{-d}\varepsilon(r)=0$, then $\var(E_k)\wc 
	V\in \Var_d(\Omega)$ and $V\mr U\times \grass{n}{d}=\var(E)$.
\end{theorem}
For any closed submanifold $\Omega$ in $\mathbb{R}^n$, and any open set
$U\subseteq \mathbb{R}^n$, if $\Omega\cap U\neq \emptyset$, we let
$\mathcal{D}(\Omega,U)$ be the collection of all deformations on  $\Omega$ in
$U$, see Definition \ref{def:deformations} for details. For any  set
$E\subseteq \Omega\cap U$, if $E$ is locally $\HM^d$ finite, we put 
$ \mathscr{E}(E)=\sup\{\HM^d(E\setminus \varphi(E))-\HM^d(\varphi(E)\setminus
	E): \varphi\in \mathcal{D}(\Omega,U)\}$.
Then $\mathscr{E}(E)\geq 0$ for all such sets, and $\mathscr{E}(E)=0$ if and
only $E$ is minimal in $U\cap \Omega$. As a consequence of Theorem 
\ref{thm:qmv}, we have the following result:
\begin{corollary}
	\label{co:convvarpre}
	Let $d$, $m$, $n$, $\Omega$, and $U$ be the same as in Theorem \ref{thm:qmv}.
	Let $\{E_k\}$ be a sequence of locally
	$\HM^d$ finite subset in $\Omega\cap U$. Suppose that $\var(E_k)\wc V\in
	\Var_d(\Omega)$ and $\mathscr{E}(E_k)\to 0$. Then there is a $d$-rectifiable
	set  $E\subseteq \Omega\cap U$, which is minimal in $U\cap \Omega$, such
	that $	V \mr U\times\grass{n}{d} =\var(E) $.
\end{corollary}

Let's continue to study the properties of generalized almost sets, which are
indeed some special quasiminimal sets. For any submanifold $\Omega\subseteq
\mathbb{R}^n$ and open set $U\subseteq \mathbb{R}^n$ with $\Omega\cap U\neq
\emptyset$, we denote by $\AM_d(\Omega\cap U,M,\varepsilon)$ the family of
quasiminimal sets $\QM_d(\Omega\cap U,U,M,\varepsilon)$. If the two 
nondecreasing functions $M$ and $\varepsilon$ satisfy certain conditions, then
$\AM_d(\Omega\cap U,M,\varepsilon)$ is compact in the sense of weak topology,
and indeed the weak topology on it is equivalent to the topology induced by
the local Hausdorff distance. See the preliminary section for the definition
of local Hausdorff distance.
\begin{theorem}\label{thm:amostminimal}
	Suppose $1\leq d<m\leq n$. Let $\Omega\subseteq \mathbb{R}^n$ be a closed
	$m$-dimensional submanifold of class $C^2$, $\partial \Omega$ be its
	boundary. Let $U\subseteq \mathbb{R}^n$ be  an open set such that $U\cap
	\Omega\neq \emptyset$ and $U\cap \partial \Omega=\emptyset$. Let
	$M:[0,\infty)\to [1,\infty)$ and $\varepsilon:[0,\infty)\to [0,\infty)$ be
	two nondecreasing functions such that $M(0+)=1$ and $\int_{0}^{r_0}r^{-d-1}
	\varepsilon(r)\ud r<\infty$ for some $r_0>0$. Then 
	\begin{itemize}
		\item $\AM_d(\Omega\cap U,M,\varepsilon)$ is compact in local Hausdorff
			distance;
		\item $\var:\AM_d(\Omega\cap U,M,\varepsilon)\to \Var_d(\Omega\cap U)$ is
			an homeomorphism between its domain and image.
	\end{itemize}
\end{theorem}

\section{Preliminary}

Let $(X,\rho)$ be a metric space, $Y\subseteq X$ a subset. We denote by
$\overline{Y}$, $\intr(X)$ and $\diam(Y)$ the closure, the interior and the 
diameter of $Y$ respectively. For any $x_0\in X$ and $r>0$, we denote by
$\oball(x_0,r)$, $\cball(x_0,r)$ and $\partial \cball(x_0,r)$ the open ball,
the closed ball and the sphere, which are centered at $x_0$ and of radius $r$,
respectively.  We denote by $\cf_c(X,\mathbb{R})$ the set of continuous
functions from $X$ to $\mathbb{R}$ with compact support, and denote by
$\mathscr{R}(X)$ the set of all Radon measures on $X$, which is always
equipped with the weak topology given by saying that $\mu_k\wc \mu$ if and 
only if $\mu_k(f)\to \mu(f)$ for any $f\in \cf_c(X,\mathbb{R})$. Indeed, the 
weak topology is the topology generated by sets
\[
	\{\mu\in \mathscr{R}(X): a<\mu(f)<b\}, \ a, b\in \mathbb{R}, \ f\in
	\cf_c(X,\mathbb{R}).
\]
If $X$ is locally compact, then for any $c>0$, $\{\mu\in \mathscr{R}(X): 
\mu(X)\leq c\}$ is compact with respect to the weak topology.

If $X$ is a locally compact metric space, $\mu, \mu_k$ are  Radon measures on
$X$, then $\mu_k\wc \mu$ if and only if for any  compactly supported bounded lower
semi-continuous function $f:X\to \mathbb{R}$ and compactly supported bounded upper
semi-continuous function $g:X\to \mathbb{R}$, it holds that 
\[
	\int_{X}f\ud \mu \leq \liminf_{k\to \infty}\int_{X}f\ud \mu_k 
\]
and 
\[
	\int_{X}g\ud \mu \geq \limsup_{k\to \infty}\int_{X}g\ud \mu_k.
\]
Indeed, the backward direction is obvious since continuous function is both
lower and upper semi-continuous; and for the forward direction, the former
inequality is following from Fatou's lemma and the fact that lower semi-continuous
function can be approximated by a  monotone increasing sequence of continuous
functions, and the later can be obtained from the former by changing the sign
of the integrand.

For any set $E\subseteq X$, the $d$-dimensional Hausdorff measure $\HM^d(E)$
is defined by 
\[
	\HM^d(E) = \lim_{\delta\to 0+} \inf \left\{ 
	\sum_{i=1}^{\infty} \diam(A_i)^d: E \subseteq \bigcup_{i=1}^{\infty}A_i,
	A_i\subseteq X, \diam(A_i)\leq \delta\right\}.
\]
A set $E\subseteq X$ is called $d$-rectifiable, if there exists a sequence of
Lipschitz mapping $f_i:\mathbb{R}^d\to X$ such that 
\[
	\HM^d \left(E\setminus \bigcup_{i=1}^{\infty} f_i(\mathbb{R}^d)\right)=0.
\]
A set $F\subseteq X$ is called purely $d$-unrectifiable if for any
$d$-rectifiable set $E\subseteq X$, $\HM^d(F\cap E)=0$. For any
$\HM^d$-measurable set $E\subseteq X$ of finite $\HM^d$ measure, it can be
decomposed as the union of a $d$-rectifiable measurable set and a purely
$d$-unrectifiable measurable set, i.e. $E=E^{irr} \sqcup E^{rec}$, the
decomposition is indeed unique up to a set of $\HM^d$ measure zero. A Radon
measure $\mu$ on $\mathbb{R}^n$ is called $d$-rectifiable if there exist a 
$d$-rectifiable set $E\subseteq \mathbb{R}^n$ and a  
Borel function $f$ such that $\mu=f \HM^d\mr E$.

Let $\mu$ be any Radon measure $\mu$ on $X$. For any $x\in
\mathbb{R}^n$, the lower and upper $d$-densities of $\mu$ at $x$ is defined by 
\[
	\Theta_{\ast}^d(\mu,x) = \liminf_{r\to 0+}\frac{\mu(\cball(x,r))}{\omega_d
r^d} \]
and 
\[
	\Theta^{\ast d}(\mu,x) = \limsup_{r\to 0+}\frac{\mu(\cball(x,r))}{\omega_d
r^d} \]
respectively, where $\omega_d$ is the $d$-dimensional Hausdorff measure of the
unity ball in $\mathbb{R}^{d}$. If the lower and upper densities are equal,
then we define the $d$-density of $\mu$ at $x$ to be 
\[
	\Theta^d(\mu,x) = \lim_{r\to 0+}\frac{\mu(\cball(x,r))}{\omega_d r^d}.
\]

If $\mu$ is a Radon measure on $\mathbb{R}^n$, $\Theta^{\ast d}(\mu,x)\in 
(0,\infty)$, and there exists a $d$-plane $T$ such that for any $\tau>0$,
\[
	\Theta^d\Big(\mu\mr \big(\mathbb{R}^n
	\setminus\mathcal{C}(T,x,\tau)\big),x\Big)=0,
\]
where $\mathcal{C}(T,x,\tau)=\{y\in \mathbb{R}^n,\dist(y-x,T)\leq \tau
|y-x|\}$, then we call $T$ an approximate tangent $d$-plane of $\mu$ at $x$.
It is quite easy to see the uniqueness if it exists, and we denote it by
$\Tan^d(\mu,x)$. For any $E\subseteq \mathbb{R}^n$ and $x\in \mathbb{R}^n$, we
define the density of $E$ and the approximate tangent $d$-plane at $x$ to be
$\Theta^d(E,x)=\Theta^d(\HM^d\mr E,x)$ and $\Tan^d(E,x) =\Tan^d(\HM^d\mr E,x)$
respectively. Suppose that $E\subseteq \mathbb{R}^n$ is
$\HM^d$-measurable, $0<t<\infty$.  If $\Theta^{\ast d}(\mu,x)\geq t$
for all $x\in E$, then $\mu(E)\geq t \HM^d(E)$. If $E$ is
$d$-rectifiable, and $\Theta_{\ast}^{d}(\mu,x)\leq t$
for all $x\in E$, then $\mu(E)\leq t \HM^d(E)$.

For any $E\subseteq \mathbb{R}^n$, if $\HM^d(E)<\infty$, then we have $2^{-d}
\leq \Theta^{\ast d}(E,x)\leq 1$ for $\HM^d$-a.e $x\in E$. If, in addition, 
$E$ is $d$-rectifiable and $\HM^d$-measurable, then $\Theta^d(E,\cdot)$
is $\HM^d$ almost equal the characteristic function of $E$; and for $\HM^d$-a.e.
$x\in E$,  $\Tan^d(E,x)$ exists and the mapping $\Tan^d(E, \cdot):E\to 
\grass{n}{d}$ is $\HM^d$-measurable.  

For any $d$-rectifiable set $E\subseteq \mathbb{R}^n$, and Lipschitz mapping 
$f:E\to \mathbb{R}^m$, we denote by $\ap Df$ and $\ap J_d f$ the
$d$-approximately differential and $d$-approximately jacobian of $f$, see 
Theorem 3.2.19, Corollary 3.2.20 and Theorem 3.2.22 in \cite{Federer:1969} for
details. For any $C^1$ mapping $f:\mathbb{R}^n\to \mathbb{R}^n$, we define the
induced mapping $f_{\sharp}:\Var_d(\mathbb{R}^n)\to \Var_d(\mathbb{R}^n)$ by 
\[
	f_{\sharp}V(\varphi)=\int_{\mathbb{R}^n\times \grass{n}{d}}\varphi(f(x),
	Df(x)T)\|\wedge_d Df(x)\circ T_{\natural}\| \ud V(x,T)
\]
for any $V\in \Var_d(\mathbb{R}^n)$ and $\varphi\in \cf_c(\mathbb{R}
^n\times\grass{n}{d},\mathbb{R})$. Let $V\in \Var_d(\mathbb{R}^n)$ be a
varifold, for any $a\in \mathbb{R}^n$, a varifold $W\in \Var_d(\mathbb{R}^n)$ 
is called a tangent varifold of $V$ at $a$, if there is a sequence $\{r_i\}$ 
of positive numbers such that $\lim_{i\to \infty}r_i = 0$, and 
\[
	W=\lim_{i\to \infty}(\scale{1/r_i})_{\sharp} V,
\]
where  $\scale{\nu}(x)=\nu\cdot(x-a)$.
We denote by $\VarTan(V,a)$ the collection of all such tangent varifolds.

Let $U\subseteq \mathbb{R}^n$ be an open set, $\{E,E_m\}$ be a sequence of
relatively closed subset in $U$. We say that $E_m$ converges to $E$ in $U$ in
local Hausdorff distance if 
\[
	\lim_{m\to \infty} d_K(E_m,E)= 0 \text{ for any compact set }K\subseteq U,
\]
where 
\[
	d_K(E_m,E)=\max\left\{\sup_{x\in E_m\cap K}\dist(x,E),\sup_{y\in E\cap
	K}\dist(y,E_m)\right\}.
\]
It is quite easy to see that the collection of all relatively closed subsets
in $U$, which is equipped with the topology induced by the local Hausdorff distance,
is compact.

\section{Deformation theorem}
In this section, we will develop a deformation theorem. By a $d$-set, we mean
a set which has positive finite many $\HM^d$ measure. For any cube
$\Delta=a+[-r,r]^k$ in $\mathbb{R}^{k}$ and $\eta>0$, we denote by $\eta
\Delta$ the cube $a+[-\eta r,\eta r]^k$, and denote by $\ell(\Delta)$ the
sidelength of $\Delta$. For any $x\in \frac{1}{2}\Delta$, let
$p_{\Delta,x}:\mathbb{R}^n\setminus \{x\}\to \partial \cball(x,\sqrt{n}
\ell(\Delta))$ be the mapping defined be 
\[
	p_{\Delta,x}(z)=x+\frac{\sqrt{n}\ell(\Delta)}{|z-x|}(z-x),\ \forall z\neq x,
\]
let $\Pi_{\Delta,x}: \Delta\setminus \{x\}\to \partial \Delta$ be the mapping
defined by 
\[
	\Pi_{\Delta,x}(z)=\{x+t(z-x):t\geq 0\}\cap \partial \Delta.
\]
\begin{lemma}\label{le:prom0}
	Let $X\subseteq \mathbb{R}^n$ be a Borel set, $E\subseteq X$ a 
	$\HM^d$-measurable set, $\Delta\subseteq \mathbb{R}^n$ a $k$-cube with
	$d+1\leq k\leq n$, $\varphi:X\to \Delta$ a continuous mapping, and $f:X\to
	[0,\infty)$ a $\HM^d$-measurable function. Suppose that $\HM^d\mr E$ is
	locally finite, $\Delta$ is centered at $x_0$  with sidelength $r>0$. Then
	there is a constant $C= C(k,d)>0$ such that for any $0<\beta<1$, we can find
	a set  $Y=Y_{\Delta,E,f,\beta}\subseteq \cball(x_0,r/4)\cap \Delta$
	satisfying that  $\HM^k(Y)\geq (1-\beta)\omega_k (r/4)^k $ and for any 
	$x\in Y$,
	\begin{equation}\label{eq:prom0}
		\int_{z\in E} \|(Dp_{\Delta,x})(\varphi(z))\|^df(z)\HM^d(z)\leq 
		\frac{C}{\beta}\int_{z\in E} f(z)\HM^d(z).
	\end{equation}
\end{lemma}
\begin{proof}
	We see that for any unit vector $v\in \mathbb{R}^k$,
	\[ Dp_{\Delta,x}(z)v=\frac{4r\sqrt{k}}{|z-x|}\left(v-\left\langle
		\frac{z-x}{|z-x|},v
		\right\rangle\frac{z-x}{|z-x|}\right)
	\]
	and $\|Dp_{\Delta,x}(z)\|=r\sqrt{k}|z-x|^{-1}$. Since there is a constant
	$c=c(k,d)>0$ such that 
	\[
		\int_{x\in \cball(x_0,r_0)\cap \mathbb{R}^k} \frac{1}{|w-x|^d}\ud \HM^k(x)
		\leq c r_0^{k-d},\ \forall w\in \mathbb{R}^k,
	\]
	we get that, by setting $\cball' = \cball(x_0,r/4)\cap \Delta$,
	\[
		\begin{aligned}
			\int_{x\in \cball'}\int_{z\in E}\|(Dp_{\Delta,x})(\varphi(z))\|^d f(z)
			\ud \HM^d(z) \ud\HM^k(x)&= \int_{z\in E}\int_{x\in \cball'}
			\frac{(r\sqrt{k})^df(z)}{|\varphi(z)-x|^d}\ud\HM^k(x)\ud \HM^d(z)\\
			&\leq (r\sqrt{k})^d \cdot c \cdot (r/4)^{k-d}\int_{z\in E} f(z)\HM^d(z) \\
			&= 4^d k^{d/2} c \cdot (r/4)^k \int_{z\in E} f(z)\HM^d(z).
		\end{aligned}
	\]
	We take $C= \omega_k^{-1}4^d k^{d/2}  c$. Then by Chebyshev's inequality,
	there is a set $Y\subseteq \cball(0,r/4)\cap \Delta$ such that
	$\HM^k(Y)\geq (1-\beta)\omega_k (r/4)^k $ and \eqref{eq:prom0} hold for any
	$x\in Y$.
\end{proof}
\begin{lemma}\label{le:prom1}
	For any $k$-cube $\Delta\subseteq \mathbb{R}^k$, $\HM^d$-measurable set 
	$E\subseteq \Delta$ and $\HM^d$-measurable function $g:E\to [0,\infty)$.  If
	$\HM^d\mr E$ is locally finite and  $\HM^{d+1}(\overline{E})=0$, then we can
	find $x_{\Delta}\in \frac{1}{2}\Delta \setminus \overline{E}$ and a
	$C^{\infty}$ mapping $q_{\Delta,x_{\Delta}}:\mathbb{R}^k\to
	\cball(x,\sqrt{k}\ell(\Delta))$ such  that, by setting $E=E^{rec}\sqcup
	E^{irr}$ and $\cball=\cball(x_{\Delta}, \dist(x_{\Delta},E)/2)$,
	$q_{\Delta,x_{\Delta}}(E^{irr})$ is purely  $d$-unrectifiable, $
	q_{\Delta,x_{\Delta}}\vert_{\Delta\setminus \cball }
	=p_{\Delta,x_{\Delta}}\vert_{\Delta\setminus \cball}$ and 
	\begin{equation}\label{eq:prom1}
		\int_{z\in E} \|Dp_{\Delta,x}(z)\|^d g(z)\HM^d(z)\leq 
		4C\int_{z\in E} g(z)\HM^d(z),
	\end{equation}
	where $C=C(k,d)>0$ is the constant in Lemma \ref{le:prom0}.
\end{lemma}
\begin{proof}
	Let $\kappa:\mathbb{R}\to [0,1]$ be a $C^{\infty}$ function such that
	$\kappa(t)=0$ for any $t\leq 0$ and $\kappa(t)=1$ for any $t\geq 1$.
	Applying Lemma \ref{le:prom0} with $\beta=1/4$, we can find $Y_{\Delta}
	\subseteq \frac{1}{2}\Delta$ such that $\HM^k(Y_{\Delta})>0$ and
	\eqref{eq:prom1} hold for any $x\in Y_{\Delta}$. By Lemma 2.2 in
	\cite{FFL:2019}, we get that for $\HM^k$-a.e. $x\in Y_{\Delta}$,
	$\Pi_{\Delta,x}(E^{irr})$ is purely $d$-unrectifiable, denote by
	$\tilde{Y}_{\Delta}$ the set of such points $x$. Since $\HM^{d+1}
	(\Delta\cap \overline{E})=0$, we have that $\HM^k( \tilde{Y}_{\Delta}
	\setminus \overline{E})>0$, pick one point $x_{\Delta}\in
	\tilde{Y}_{\Delta}\setminus \overline{E}$, put $r_{\Delta}=\dist(x_{\Delta},
	\overline{E})$, and define the  $C^{\infty}$ mapping $q_{\Delta,x}:
	\mathbb{R}^n\to \cball(x,\sqrt{n} \ell(\Delta))$ by
	\[ q_{\Delta,x}(z)=\begin{cases}

			x+\kappa\left(2r_{\Delta}^{-1}(z-x)\right)\left(p_{\Delta,x}(z)-x\right),
			&z\neq x,\\ x,& z=x.
		\end{cases}
	\]
	Then $q_{\Delta,x_{\Delta}}\vert_{\Delta\setminus \cball }
	=p_{\Delta,x_{\Delta}}\vert_{\Delta\setminus \cball}$ and
	$q_{\Delta,x_{\Delta}}(E^{irr})=p_{\Delta,x_{\Delta}}(E^{irr})$ is purely
	$d$-unrectifiable.
\end{proof}
\begin{lemma}\label{le:prom2}
	Let $\Delta\subseteq \mathbb{R}^k$ be a $k$-cube.	For any $\HM^d$-measurable
	set $E\subseteq \Delta$ and $\HM^d$-measurable function $g:\Delta\to
	\mathbb{R}$, if $\HM^d\mr E$ is locally finite, then  \[ \int g \ud \HM^d
		\mr \Pi_{\Delta,x}(E)\leq 2^{d}k^{d/2} \int (g\circ
	\Pi_{\Delta,x})\|Dp_{\Delta,x}\|^d \ud \HM^d \mr E \]
\end{lemma}
\begin{proof}
	Let $f_{\Delta,x}:\partial \cball(x,\sqrt{k}\ell(\Delta)) \to \partial
	\Delta$ be the mapping defined by 
	\[ f_{\Delta,x}(z)=\{x+t(z-x):t\geq 0\}\cap \partial \Delta,\ \forall z\in
		\partial 
		\cball(0,\sqrt{n}\ell(\Delta)).
	\]
	Then $f_{\Delta,x}$ is biLipschitz and $ f_{\Delta,x}\circ p_{\Delta,x}=
	\Pi_{\Delta,x} $. Since $\cball(x,\ell(\Delta)/2)\subseteq
	\Delta\subseteq \cball(x,\sqrt{k}\ell(\Delta))$, we get that
	$\Lip(f_{\Delta,x})\leq 2\sqrt{k}$. Since $\HM^d(f_{\Delta,x}(Z))\leq
	\Lip(f_{\Delta,x})^d\HM^d(Z)$ for any set $Z\subseteq \partial 
	\cball(x,\sqrt{k}\ell(\Delta))$, we get that 
	\[
		\int g \ud \HM^d \mr f_{\Delta,x}(Z)\leq
		\Lip(f_{\Delta,x})^d\int g\circ
		f_{\Delta,x} \ud \HM^d \mr Z.
	\]
	Since $p_{\Delta,x}$ is a $C^{\infty}$ mapping and
	$\Pi_{\Delta,x}=f_{\Delta,x}\circ p_{\Delta,x}$, we get that
	\[
		\begin{aligned}
			\int g \ud \HM^d \mr \Pi_{\Delta,x}(E)&\leq \Lip(f_{\Delta,x})^d\int
			g\circ f_{\Delta,x} \ud \HM^d \mr p_{\Delta,x}(E)\\
			&\leq \Lip(f_{\Delta,x})^d\int (g\circ
			f_{\Delta,x}\circ p_{\Delta,x}) \|Dp_{\Delta,x}\|^d \ud \HM^d \mr E\\
			&\leq 2^{d}k^{d/2} \int (g\circ
			\Pi_{\Delta,x})\|Dp_{\Delta,x}\|^d \ud \HM^d \mr E.
		\end{aligned}
	\]

\end{proof}
\begin{lemma}\label{le:prom4}
	Let $X\subseteq \mathbb{R}^m$ be any set. Let $Y\subseteq \mathbb{R}^k$ be a
	convex compact set, $\tilde{f}: X\to \partial Y$ be a Lipschitz mapping. Then
	we can find Lipschitz mapping $f:\mathbb{R}^m\to Y$ such that $f\vert_{X}=
	\tilde{f}$ and $\Lip(f)=\Lip(\tilde{f})$. Moreover any Lipschitz mapping,
	which is defined on the $k$-skeleton of an $n$-cube $\Delta\subseteq
	\mathbb{R}^{n}$, and which maps each $k$-faces of $\Delta$ to itself, admit
	a Lipschitz extension which maps each $i$-faces of $\Delta$ to itself,
	$k\leq i\leq n$.  
\end{lemma}
\begin{proof}
	Let $g:\mathbb{R}^m\to \mathbb{R}^k$ be a Lipschitz extension of $\tilde{f}$
	such that $\Lip(g)=\Lip(\tilde{f})$. Since $Y$ is
	convex, for any $z\in \mathbb{R}^k$, there is a unique point $x_z\in Y$
	such that $\dist(z,Y)=|z-x_z|$. We denote by $\rho:\mathbb{R}^k\to Y$ the 
	mapping defined by $\rho(z)=x_z$, and then define 
	\[
		f=\rho\circ g.
	\]
	Since $\Lip(\rho)=1$ and $\rho\vert_{Y}=\id_{Y}$, we see that $f$ is a
	Lipschitz extension of $\tilde{f}$ with the same Lipschitz constant. The
	moreover part then follows from induction.

\end{proof}
\begin{lemma}
	\label{le:prom3}
	Let $\Delta\subseteq \mathbb{R}^k$ be a $k$-cube. For any $x\in
	\frac{1}{2}\Delta$, setting $\cball=\cball(x,\sqrt{k}\ell(\Delta))$, there
	exists a Lipschitz mapping $f_{\Delta,x}:\cball\to \Delta$ such that
	$f_{\Delta,x}\vert_{\partial \cball}:\partial
	\cball\to \partial\Delta$ is biLipschitz, $f_{\Delta,x}\circ
	p_{\Delta,x}\vert_{\Delta}=\Pi_{\Delta,x}$ and 
	\begin{equation}\label{eq:prom30}
		\Lip(f_{\Delta,x})=\Lip(f_{\Delta,x}\vert_{\partial \cball })
		\leq 2^{d}k^{d/2}.
	\end{equation}
\end{lemma}
\begin{proof}
	For any $z\in \partial \cball$, we define
	\[
		h(z)=\{x+t(z-x):t\geq 0\}\cap \partial \Delta.
	\]
	Then $g:\partial \cball\to \partial \Delta$ is a biLipschitz mapping with
	$\Lip(h)\leq 2^dk^{d/2}$. Let $f_{\Delta,x}: \cball\to \Delta$ be a
	Lipschitz extension of $h$ with $\Lip(\tilde{h})=\Lip(h)$ as in Lemma
	\ref{le:prom4}. Then we see that
	$f_{\Delta,x}\circ p_{\Delta,x}\vert_{\Delta}=h\circ
	p_{\Delta,x}\vert_{\Delta}=\Pi_{\Delta,x} $ and \eqref{eq:prom30} holds.
\end{proof}
\begin{definition}
	Let $\mathscr{F}$ be a collection of finitely many $n$-cubes in
	$\mathbb{R}^n$, let $\mathscr{F}_m$ be the collection all
	$m$-faces of cubes in $\mathscr{F}$. We say that $\mathscr{F}$ is 
	admissible if 
	\begin{itemize}
		\item for any $\Delta_0,\Delta_1\in \mathscr{F}_m$, either
			$\mathring{\Delta}_0\cap \mathring{\Delta}_1=\emptyset$,
			$\Delta_0\subseteq \Delta_1$ or $\Delta_1\subseteq \Delta_0$, where
			$\mathring{\Delta}_i=\Delta_i\setminus \cup \mathscr{F}_{m-1}$, and
		\item for any $\Delta_0,\Delta_1\in \mathscr{F}_m$, if
			$\Delta_1\subsetneq \Delta_0$, then there exist $\{\Delta_i\}_{2\leq
			i\leq I}\subseteq \mathscr{F}_m$ such that $\Delta_0=\cup_{1\leq i\leq
			I} \Delta_i$ and $\mathring{\Delta}_i\cap \mathring{\Delta}_j=\emptyset$ for any $1\leq
			i<j\leq I$.
	\end{itemize}
\end{definition}
\begin{theorem}[Deformation theorem]\label{thm:ffp}
	Let $\mathscr{F}$ be a collection of $n$-cubes in $\mathbb{R}^{n}$ which is
	admissible. Set $D=\cup \mathscr{F}$ and $\mathscr{F}_m'=\{\Delta\in 
	\mathscr{F}_m:\mathring{\Delta }\subseteq \intr(D)\}$ for $0\leq m\leq n$. There exists $c_0=c_0(n,d)\geq 1$ such that
	for any $d$-set $E\subseteq \mathbb{R}^n$ with $\HM^{d+1}\big(\overline{E}\cap
	\intr(D)\big)<\infty$, setting $E\cap \intr(D)=E^{rec}\sqcup E^{irr}$ such that
	$E^{rec}$ is $d$-rectifiable and $\HM^d$-measurable, $E^{irr}$ is purely 
	$d$-unrectifiable and $\HM^d$-measurable, we can find a Lipschitz mapping 
	$\phi:\mathbb{R}^n\to \mathbb{R}^n$, an upper semi-continuous function $\lambda:
	\mathbb{R}^n\to [0,\infty)$, $\mathscr{S}\subseteq 
	\mathscr{F}_{d}'$ and an open set $W\subseteq \mathbb{R}^n$ such that
	\begin{itemize}
		\item $\phi(\Delta)\subseteq\Delta$ for any $\Delta\in \mathscr{F}_m$,
			$\phi\vert_{\mathbb{R}^n\setminus D}= \id_{\mathbb{R}^n\setminus D}$,
			$\overline{E}\subseteq W$,
		\item  $(\cup
			\mathscr{S})\cap \intr(D)\subseteq\phi\big(\overline{E}\big)\cap
			\intr(D)\subseteq \phi(W)\cap
			\intr(D)\subseteq \cup \mathscr{F}_d'$, $\phi(W)\cap
			\intr(D)\setminus \cup\mathscr{S}\subseteq \cup \mathscr{F}_{d-1}'$,
		\item $\HM^d\big(\phi(E^{irr})\cap \intr(D)\big)=0$, $\lambda(x)=1$ for any $x\in
			\mathbb{R}^n\setminus D$,
		\item for any $\HM^d$-measurable $d$-set $Z\subseteq W$ and 
			$\HM^d$-measurable function $g:\mathbb{R}^n\to [0,\infty)$, 
			\begin{equation}\label{eq:ffp1}
				\int g \ud \HM^d \mr \phi(Z) \leq \int (g\circ \phi) \cdot \lambda \ud 
				\HM^d \mr Z,
			\end{equation}
		\item for any $\mathscr{K}\subseteq \mathscr{F}_n$, setting $K=\cup
			\mathscr{K}$ and $A=\cup\{\Delta\in \mathscr{F}_n: \Delta\cap \partial
			D\neq \emptyset\}$, we have that 
			\begin{equation}\label{eq:ffp2}
				\HM^d(\phi(E\cap K))\leq \int_{E\cap K}
				\lambda(x) \ud \HM^d(x)\leq c_0 \Big(\HM^d(E^{rec}\cap
				K)+\HM^d(E^{irr}\cap K\cap A)\Big).
			\end{equation}
	\end{itemize}
	In particular, if $E\cap \intr(D)$ is relatively closed in
	$\intr(D)$ and 
	\begin{equation}\label{eq:ffp3}
		\HM^d(E^{rec}\cap \intr(D))\leq c_0^{-1}
		\min\left\{\ell(\Delta)^d:\Delta\in \mathscr{F}_n'\right\},
	\end{equation}
	then 
	\begin{equation}\label{eq:ffp4}
		\phi(W)\cap \intr(D)\subseteq \cup \mathscr{F}_{d-1}'.
	\end{equation}
\end{theorem}
\begin{proof}
	For any $\Delta\in \mathscr{F}_n$, let $x_{\Delta}$ and
	$q_{\Delta,x_{\Delta}}$ be as in Lemma \ref{le:prom1}, let
	$f_{\Delta,x_{\Delta}}$ be as in Lemma \ref{le:prom3} with $x=x_{\Delta}$.
	We define the mapping $\phi_1:\mathbb{R}^n\to \mathbb{R}^n$ and
	$\psi_1:\mathbb{R}^n\setminus\{x_{\Delta}:\Delta\in \mathscr{F}_n\}\to
	\mathbb{R}^n$ by
	\[
		\phi_1(z)=\begin{cases}
			f_{\Delta,x_{\Delta}}\circ q_{\Delta,x_{\Delta}}(z),&z\in \Delta,\ \Delta\in \mathscr{F}_n,\\
			z,& \text{ otherwise},
		\end{cases}
	\]
	and 
	\[
		\psi_1(z)=\begin{cases}
			f_{\Delta,x_{\Delta}}\circ p_{\Delta,x_{\Delta}}(z),&z\in
			\Delta\setminus\{x_{\Delta}\},\ \Delta\in \mathscr{F}_n,\\
			z,& \text{ otherwise}.
		\end{cases}
	\]
	Then we see that $\phi_1$ is Lipschitz, and by setting
	$r_{\Delta}=\frac{1}{2}\dist(x_{\Delta},E)$, $\psi_1$ coincide with $\phi_1$ on the set
	\[
		U_1=\mathbb{R}^n\setminus 
		\medcup_{\Delta\in \mathscr{F}_n}\cball\left(x_{\Delta},r_{\Delta}\right).
	\]
	Put $E_0=E$, $E_1=\psi_1(E_0)$, and $E_1\cap \intr(D)=E_1^{rec}\sqcup 
	E_1^{irr}$. Since $\phi_1$ is Lipschitz, we see that
	\[
		\overline{E}_1=\overline{\psi_1(E)}=\overline{\phi_1(E)}=
		\phi_1\Big(\overline{E}\Big),
	\]
	and $\HM^{d+1}(\overline{E}_1)=0$.
	Similarly, for any $\Delta\in \mathscr{F}_{n-1}$ with
	$\mathring{\Delta}\subseteq \intr(D)$, we can find $x_{\Delta}\in
	\frac{1}{2}\Delta\setminus \overline{E}_1$, such that
	$\Pi_{\Delta,x_{\Delta}}(E_1^{irr}\cap \Delta)$
	is purely $d$-unrectifiable. Define the mapping $\psi_2$
	by 
	\[
		\psi_2(z)=\begin{cases}
			f_{\Delta,x_{\Delta}}\circ
			p_{\Delta,x_{\Delta}}(z)=\Pi_{\Delta,x_{\Delta}}(z),&z\in
			\Delta\setminus\{x_{\Delta}\},\ \Delta\in \mathscr{F}_{n-1}',\\
			z,& z\in \mathbb{R}^n\setminus \intr(D).
		\end{cases}
	\]
	We see that $\psi_2$ is not define on whole space $\mathbb{R}^n$, which is only 
	defined on 
	\[
		\dmn(\psi_2)=\left(\mathbb{R}^n\setminus
		\intr(D)\right)\cup \medcup\left\{\Delta\setminus \{x_{\Delta}\}: \Delta\in 
		\mathscr{F}_{n-1}'\right\},
	\]
	but fortunately it is well defined on $\psi_1(U_1)$ at least, and
	$\im(\psi_2)\cap \intr(D)\subseteq \cup\{\Delta:\Delta\in \mathscr{F}_{n-1}\}$.
	Put \[
		r_{\Delta}=\frac{1}{2}\dist(x_{\Delta},E_1) 
		\text{ and }
		U_2=(\mathbb{R}^n\setminus
		\intr(D)) \cup \medcup\left\{\Delta\setminus
		\cball(x_{\Delta},r_{\Delta}): \Delta\in 
		\mathscr{F}_{n-1}'\right\}.
	\]
	We see that
	$\psi_2\vert_{U_2}$ is Lipschitz. Let $\phi_2:\mathbb{R}^n\to \mathbb{R}^n$
	be a Lipschitz extension of $\psi_2\vert_{U_2}$, whose existence is 
	guaranteed by Lemma \ref{le:prom3} and Lemma \ref{le:prom4}, such that
	$\psi_2(\Delta)\subseteq \Delta$ for any $\Delta\in \mathscr{F}_n$.  

	By induction, we can define $\psi_i$, $U_i$, $E_i$ and $\phi_i:\mathbb{R}^n\to
	\mathbb{R}^n$ for $1\leq i\leq n-d$ in a 
	similar way, which indeed are satisfying that
	\[
		\dmn(\psi_i)=(\mathbb{R}^n\setminus
		\intr(D))\cup \medcup\{\Delta\setminus \{x_{\Delta}\}: \Delta\in 
		\mathscr{F}_{n-i+1}'\},
	\]
	\[
		\psi_i(z)=\begin{cases}
			\Pi_{\Delta,x_{\Delta}}(z),&\
			\Delta\in \mathscr{F}_{n-i+1}',\\
			z,&z\in \mathbb{R}^n\setminus \intr(D),
		\end{cases}
	\]
	\[
		E_i=\psi_i(E_{i-1}),\ r_{\Delta}=\dist(x_{\Delta},E_{i-1}),
	\]
	\[
		U_i=\left(\mathbb{R}^n\setminus
		\intr(D)\right)\cup \medcup\left\{\Delta\setminus
		\cball(x_{\Delta},r_{\Delta}): \Delta\in 
		\mathscr{F}_{n-i+1}'\right\},
	\]
	$\psi_i(E_{i-1}^{irr}\cap \Delta)$ is purely $d$-unrectifiable for any
	$\Delta\in \mathscr{F}_{n-i+1}$,
	$\psi_i\vert_{U_i}$ is Lipschitz, and $\phi_i$ is a Lipschitz extension of
	$\psi_i\vert_{U_i}$ such that $\phi_i(\Delta)\subseteq \Delta$ for any
	$\Delta\in \cup_{n-i+1\leq j\leq n} \mathscr{F}_{j}$
	Since $E_{n-d}=\psi_{n-d}(E_{n-d-1})$ and $\im(\psi_{n-d})\cap
	\intr(D)\subseteq \cup \{\Delta:\Delta\in \mathscr{F}_{d}\}$, we get that
	\[
		E_{n-d}\cap \intr(D)\subseteq \cup \{\Delta:\Delta\in \mathscr{F}_{d}'\}.
	\]
	For any $\Delta\in \mathscr{F}_{d}'$, if $\Delta\setminus
	\overline{E}_{n-d}\neq \emptyset$, then we pick one point $x_{\Delta}\in
	\Delta\setminus \overline{E}_{n-d} $. Let $\phi_{n-d+1}:\mathbb{R}^n\to
	\mathbb{R}^n$ be a Lipschitz mapping such that
	$\phi_{n-d+1}\vert_{\Delta}=\id_{\Delta}$ if $\Delta\in \mathscr{F}_{d}'$ and
	$\Delta\setminus \overline{E}_{n-d}=\emptyset$ or $\Delta\subseteq \partial
	D$; $\phi_{n-d+1}\vert_{\Delta}=q_{\Delta,x_{\Delta}}$ if $\Delta\setminus
	\overline{E}_{n-d}\neq \emptyset$. Put $E_{n-d+1}=\phi_{n-d+1}(E_{n-d})$,
	$r_{\Delta}=\dist(x_{\Delta},\overline{E}_{n-d})$, and 
	\[
		U_{n-d+1}= \Big(\mathbb{R}^n\setminus \intr(D)\Big) \cup \Big(\cup
		\mathscr{F}_{d}'\Big)\setminus \Big( \cup \big\{\cball(x_{\Delta}, r_{\Delta}):
		\Delta\in \mathscr{F}_d', \Delta\setminus
		\overline{E}_{n-d}\neq\emptyset\big\}\Big).
	\]
	Put $W_{n-d+1}=U_{n-d+1}$, and $W_i=U_i\cap \phi_i^{-1}(W_{i+1})$ for $1\leq
	i\leq n-d$. Since $U_{i}$ is open in $\big(\mathbb{R}^n\setminus \intr(D)\big) 
	\cup \big(\cup \mathscr{F}_{n-i+1}'\big)$, $1\leq i\leq n-d+1$, we get that 
	$W_i$ is open in $U_i$. Since $U_1$ is open in $\mathbb{R}^{n}$, we get that
	$W_1$ is open in $\mathbb{R}^{n}$. Put $\lambda_{n-d+1}\equiv 1$, and define
	$\lambda_i:\cup \mathscr{F}_{n-i+1}\to \mathbb{R}$, $2\leq i\leq n-d$, by 
	\[
		\lambda_i(z)=\begin{cases}
			2^d(n-i+1)^{d/2}\|Dq_{\Delta,x_{\Delta}}(z)\|^d\cdot
			\lambda_{i+1}(\phi_{n-i+1}(z)), &z\in
			\mathring{\Delta}, \Delta\in \mathscr{F}_{n-i+1}',\\
			2^d(n-i+1)^{d/2}\cdot 4^d(n-i+1)^{d/2} \cdot\lambda_{i+1}(z),& 
			z\in \Delta\setminus \mathring{\Delta},\Delta\in \mathscr{F}_{n-i+1}',\\
			1,& z\in \cup \mathscr{F}_{n-i+1}\setminus \cup \mathscr{F}_{n-i+1}'\\
		\end{cases}
	\]
	Let $\lambda_1:\mathbb{R}^n\to \mathbb{R}$ be defined by 
	\[
		\lambda_1(z)=\begin{cases}
			2^dn^{d/2}\|Dq_{\Delta,x_{\Delta}}(z)\|^d\cdot
			\lambda_{2}(\phi_1(z)), &z\in
			\mathring{\Delta}, \Delta\in \mathscr{F}_{n},\\
			2^dn^{d/2}\cdot 4^dn^{d/2} \cdot\lambda_{2}(z),& 
			z\in \Delta\setminus \mathring{\Delta},\Delta\in \mathscr{F}_{n},\\
			1, & z\in \mathbb{R}^n\setminus D.
		\end{cases}
	\]

	Define $\phi=\phi_{n-d+1}\circ \phi_{n-d}\circ \cdots \circ \phi_1$,
	$W=W_1$ and $\lambda=\lambda_1$. Then $\phi\vert_{\mathbb{R}^n\setminus
	D}=\id_{\mathbb{R}^n\setminus D}$, $\phi(\Delta) \subseteq \Delta $ for
	any $\Delta\in \mathscr{F}_m$, $0\leq m\leq n$, $\overline{E}\subseteq W$,
	$\phi(E)=E_{n-d+1}$ and $\phi\big(\overline{E}\big)=\overline{E}_{n-d+1}$. 
	Put $\mathscr{S}=\{\Delta\in \mathscr{F}_d:\Delta\setminus
	\overline{E}_{n-d+1}=\emptyset\}$. Then $\cup \mathscr{S}\subseteq
	\overline{E}_{n-d+1}$ and $\phi(W)\subseteq
	\phi_{n-d+1}(U_{n-d+1})$, thus $\phi(W)\cap \intr(D)\subseteq \cup
	\mathscr{F}_d'$ and $\phi(W)\cap \intr(D)\setminus \cup
	\mathscr{S}\subseteq \cup \mathscr{F}_{d-1}'$. By our construction of
	$\psi_i$, we see that  $ \psi_{n-d}\circ \cdots \circ \psi_1(E^{irr})$ is 
	purely $d$-unrectifiable, but $\intr(D)\cap \psi_{n-d}\circ \cdots \circ
	\psi_1(E^{irr})$ is contained in $\cup \mathscr{F}_{d}'$, thus $\intr(D)\cap
	\psi_{n-d}\circ \cdots \circ \psi_1(E^{irr})$ must be a set of $\HM^d$
	measure 0. So we get that $\HM^d(\intr(D)\cap \phi(E^{irr}))=0$. Indeed,
	$\lambda\vert_{\mathbb{R}^n\setminus D}\equiv 1$ and \eqref{eq:ffp1}
	clearly follows from the construction of $\lambda$ and Lemma \ref{le:prom2}.
	For any $\Delta\in \mathscr{F}_m'$, $2\leq m\leq n-d$, since $x_{\Delta}\in
	\frac{1}{2}\Delta$, we see that 
	\[
		\|Dq_{\Delta,x_{\Delta}}(z)\|=\frac{\sqrt{m}\ell(\Delta)}{|z-x_{\Delta}|}\leq
		4\sqrt{m},\ \forall z\in \Delta\setminus \mathring{\Delta},
	\]
	thus the functions $\lambda_i$, $1\leq i\leq n-d$, are upper
	semi-continuous. In particular, $\lambda=\lambda_1$ is upper
	semi-continuous. For any $1\leq i\leq n-d$ and $\Delta\in \mathscr{F}_{n-i+1}'$,
	setting $\phi_0=\id_{\mathbb{R}^n}$, since $\HM^d(E^{irr}\cap
	\mathring{\Delta})=0$, by Lemma \ref{le:prom0}, there is a 
	constant $\zeta_i>0$, which only depends on $n-i+1$ and $d$, such that 
	\[
		\int_{\mathring{\Delta}\cap E}\lambda_i(\phi_{i-1}\circ\cdots\circ
		\phi_0(z))\ud \HM^d(z)=\int_{\mathring{\Delta}\cap E^{rec}}
		\lambda_i(\phi_{i-1}\circ\cdots\circ \phi_0(z))\ud \HM^d(z)\leq 
		\zeta_i \HM^d(E^{rec}\cap \mathring{\Delta});
	\]
	if $\Delta\in \mathscr{F}_{n-i+1}\setminus \mathscr{F}_{n-i+1}'$, by Lemma 2.1, we
	only get that 
	\[
		\int_{\mathring{\Delta}\cap E}\lambda_i(\phi_{i-1}\circ\cdots\circ
		\phi_0(z))\ud \HM^d(z)\leq \zeta_i \HM^d(E\cap \mathring{\Delta}),
	\]
	hence \eqref{eq:ffp2} holds with $c_0=\max\{1,\zeta_1\}$.

\end{proof}
\begin{lemma}
	\label{le:ffpa}
	Let $c_0=c_0(n,d)\geq 1$ be the constant in Theorem \ref{thm:ffp}.
	There is a constant $c_1=c_1(n,d)>0$ such that 
	for any $d$-set $E\subseteq \mathbb{R}^n$ and $0<r<\rho$ with $\HM^{d+1}
	(\overline{E}\cap \oball(x,\rho)) <\infty$, there is a 
	Lipschitz mapping $\phi:\mathbb{R}^n\to \mathbb{R}^n$ such that
	$\phi(\oball(x,\rho))\subseteq \oball(x,\rho)$,
	$\phi\vert_{\mathbb{R}^n\setminus
	\oball(x,\rho)}=\id_{\mathbb{R}^n\setminus \oball(x,\rho)}$, 
	\begin{equation}\label{eq:ffpa1}
		\HM^d(\phi(E\cap \oball(x,\rho)))\leq c_0\left(\HM^d(E\cap
		\ann(x,r,\rho))+\HM(E^{rec}\cap \oball(x,\rho))\right),
	\end{equation}
	and
	\begin{equation}\label{eq:ffpa2}
		\HM^d(\phi(E\cap \oball(x,\rho)))\leq c_0\HM^d(E\cap \ann(x,r,
		\rho))+\frac{c_1\rho^n}{(\rho-r)^{n-d}},
	\end{equation}
	where $\ann(x,r,\rho)=\oball(x,\rho)\setminus \cball(x,r)$, $E^{rec}$ is the
	$d$-rectifiable part of $E$. Moreover, if $r>\rho/4$ and 
	\begin{equation}\label{eq:ffpa3}
		\HM^d(E^{rec}\cap \oball(x,\rho))<c_0^{-1}\big(3 \sqrt{n}\big)^{d-n}(\rho-r)^d,
	\end{equation}
	then we can require $\phi$ satisfying that
	\begin{equation}\label{eq:ffpa4}
		\HM^d(\phi(E\cap \oball(x,\rho)))\leq c_0\HM^d(E\cap \ann(x,r,
		\rho)).
	\end{equation}
\end{lemma}
\begin{proof}
	Put $\delta=(\rho-r)/(3\sqrt{n})$. For any $i_1, \cdots, i_n\in \mathbb{Z}$,
	we denote by $C_{i_1,\cdots,i_n}$ the cube defined by $(\delta i_1,\cdots,
	\delta i_n)+[0,\delta]^n$. Let $\mathscr{F}$ be the collection of all cubes
	$C_{i_1,\cdots,i_n}\subseteq \oball(x,\rho)$, and let
	$\mathscr{K}=\{\Delta\in \mathscr{F}:\Delta\cap \cball(x,r)\neq \emptyset\}$.
	By Theorem \ref{thm:ffp}, we can find Lipschitz mapping
	$\phi:\mathbb{R}^n\to \mathbb{R}^n$ and $\mathscr{S}\subseteq
	\mathscr{F}_{d}'$ such that $\phi(\Delta)\subseteq \Delta$
	for any $\Delta\in \mathscr{F}$, $\phi(x)=x$ for any $x\in \mathbb{R}^n\setminus
	\cup \mathscr{F}$, $\phi(E)\cap \intr(\cup \mathscr{F}_n)\subseteq \cup
	\mathscr{F}_d$, $\phi(\overline{E})\cap \intr(\cup \mathscr{F}_n)\setminus
	\cup \mathscr{S}\subseteq \mathscr{F}_{d-1}$ and 
	\[
		\HM^d(\phi(E\cap (\cup \mathscr{K})))\leq c_0\HM^d(E^{rec}\cap (\cup
		\mathscr{K})).
	\]
	Immediately, we get that \eqref{eq:ffpa1} holds since 
	\[
		\HM^d(\phi(E\cap \oball(x,\rho)\setminus \cup \mathscr{K}))\leq
		c_0\HM^d(E\cap \oball(x,\rho)\setminus \cup \mathscr{K})\leq
		c_0\HM^d(E\cap \ann(x,r,\rho)).
	\]
	Since $\phi(E\cap \cball(x,r))\subseteq \cup \mathscr{F}_{d}'$, and the
	number of cubes in $\mathscr{F}$ is no more than $\omega_n \rho^d/\delta^n$,
	we get that \eqref{eq:ffpa2} hold for some constant $c_1$ which only depends
	on $n$ and $d$. Since $\phi(E)\cap (\cup \mathscr{F}_n')\subseteq \cup \mathscr{S}
	$ and $\HM^d(\phi(E)\cap (\cup \mathscr{F}_n'))\leq c_0\HM^d(E^{rec}\cap
	\oball(x,\rho))<\delta^d$, but each element of $\mathscr{S}$ has $\HM^d$
	measure 0 or $\delta^d$, we get that $\HM^d(\phi(E)\cap (\cup
	\mathscr{F}_n'))=0$, and \eqref{eq:ffpa4} holds.

\end{proof}
\section{Quasiminimal sets on $\mathbb{R}^n$}
\begin{definition} \label{def:deformations}
	For any set $\Omega\subseteq \mathbb{R}^n$ and any open set $U\subseteq
	\mathbb{R}^n$, if $\Omega\cap U$ is nonempty and relatively closed in $U$,
	we denote by $\mathcal{D}(\Omega, U)$ the collection of all Lipschitz 
	mappings $\varphi:\Omega\to \Omega$ which is homotopic to $\id_{\Omega}$,
	and that the closure of $\{x\in \Omega: \varphi(x)\neq x\}\cup \varphi( 
	\{x\in \Omega: \varphi(x)\neq x\})$ is compact and
	contained in $U$. Indeed, we call such a Lipschitz mapping $\varphi$ a
	deformations on $\Omega$ in $U$.
\end{definition}

\begin{definition}\label{def:quasimimimal}
	Let $\Omega,U$ be as in Definition \ref{def:deformations}. For any
	nondecreasing functions $M:[0,\infty)\to [1,\infty]$ and $\varepsilon:
	[0,\infty )\to [0,\infty]$, we define $\QM_d(\Omega,U,M,\varepsilon)$ to be 
	the collection of sets $E\subseteq \Omega$ such that $\HM^d\mr E$ is locally
	finite, $\HM^d(E\cap \oball(x,r))>0$ for any $x\in E\cap U$ and $r>0$, 
	$E\cap U$ is relatively closed in $U$, and for any $\varphi\in
	\mathcal{D}(\Omega,U)$,
	\[
		\HM^d(E\cap W_{\varphi})\leq M(r)\HM^d(\varphi(E\cap
		W_{\varphi}))+\varepsilon(r),
	\]
	where $W_{\varphi}=\{x\in \Omega: \varphi(x)\neq x\}$, $r=\diam(W_{\varphi}\cap \varphi(W_{\varphi}))$.
\end{definition}

If $M\equiv 1$ and $\varepsilon\equiv 0$, then elements in
$\QM_d(\Omega,U,M, \varepsilon)$ are called minimal in $\Omega\cap U$.

If $M\equiv 1$, and $\varepsilon(r)=h(r)r^d$, where $h$ is a gauge function, 
then elements in $\QM_d(\mathbb{R}^n,U,M, \varepsilon)$ are called almost 
minimal sets in $U$ with gauge function $h$,  see Definition 4.3 in \cite{David:2009}.

If there is a $\delta>0$ such that $M(r)\equiv M\geq 1$  and $\varepsilon(r)
\equiv 0$ for $0<r<\delta$, then elements in $\QM_d(\mathbb{R}^n,U,M, \varepsilon)$
are called $(U,M,\delta)$-quasiminimal, see Definition 2.4 in \cite{David:2003}.

If there is a $\delta>0$ and a constant $h\in (0,1)$ such that $M(r)\equiv 
M\geq 1$  and $\varepsilon(r) =h r^d$ for $0<r<\delta$, then elements in
$\QM_d(\mathbb{R}^n,U,M, \varepsilon)$ are called generalized quasiminimal, see
Definition 2.10 in \cite{David:2009}.

If $\delta>0$ and $\varepsilon:(0,\delta)\to [0,\infty)$ is a nondecreasing
function such that $\varepsilon(0+)=0$, then any element in
$\QM_d(\mathbb{R}^n,U,1+\varepsilon(r),0)$ is called
$(\mathbf{M},\varepsilon,\delta)$-minimal set, see for example Chapter 11 in
\cite{Morgan:2009}.

We do not assume previously $\varepsilon(0+)=0$ in our definition. It is easy
to see from the definition that 
\[
	\QM_d(\Omega,U,M,\varepsilon)\subseteq \QM_d(\Omega',U',M',\varepsilon'),
\]
if $\Omega'\subseteq \Omega$, $U'\subseteq U$, $M\leq M'$ and $\varepsilon\leq
\varepsilon'$.
\begin{lemma}
	\label{le:irr}
	For any $E\in \QM_d(\mathbb{R}^n,U,M,\varepsilon)$ and open set
	$O\subseteq \mathbb{R}^n$, we have that
	\[
		\HM^d(E^{irr}\cap O)\leq \varepsilon(\diam O),
	\]
	where $E^{irr}$ is the purely $d$-unrectifiable part of $E\cap U$.
\end{lemma}
\begin{proof}
	We Write $E\cap U=E^{irr}\sqcup E^{rec}$, where $E^{rec}$ is
	$d$-rectifiable and $\HM^d$-measurable, $E^{irr}$ is purely
	$d$-unrectifiable and $\HM^d$-measurable. Then for $\HM^d$-a.e. $x\in
	E^{irr}$, 
	\[
		\Theta^d(E^{rec},x)=0 \text{ and }\Theta^{\ast d}(E^{irr},x)\geq
		2^{-d},
	\]
	and we denote by $E_0$ the collection of such point $x$. For any $\tau>0$,
	and $x\in E_0$, we can find a sequence of decreasing positive numbers
	$\{\rho_{x,m}\}$ such that $\rho_{x,m}\to 0$ as $m\to \infty$, $\HM^d(E\cap
	\partial \cball(x,\rho_{x,m}))=0$,
	\[
		\frac{\HM^d(E^{rec}\cap \cball(x,r))}{\omega_d r^d}\leq \tau, \ \forall
		0<r\leq \rho_{x,1}
	\]
	and 
	\[
		\frac{\HM^d(E^{irr}\cap \cball(x,\rho_{x,m}))}{\omega_d r^d}>\frac{1}{2^{d+1}}.
	\]
	We see that $\{\cball(x,\rho_{x,m}):x\in
	E_0, m\geq 1, \rho_{x,m}<\dist(x,\mathbb{R}^n\setminus O)\}$ is a Vitali
	covering of $E_0\cap O$. By Vitali covering theorem, we can find finite many
	disjoint balls $\{\cball_i\}_{1\leq i\leq m}$ such that 
	\[
		\HM^d\left(E_0\setminus \cup_{i=1}^m \cball_i\right)\leq \tau
		\HM^{d}(E_0).
	\]
	For any $0<r<\rho<\dist(x,\mathbb{R}^n\setminus O)$, by Lemma 
	\ref{le:ffpa}, there exists Lipschitz mapping $\phi_x:\mathbb{R}^n\to
	\mathbb{R}^n$ such that $\phi_x(\cball(x,\rho))\subseteq
	\cball(x,\rho)$, $\phi_x(y)=y$ for $y\notin \oball(x,\rho)$ and 
	\[
		\HM^d(\varphi (E\cap \oball(x,\rho)))\leq c_0 \big(\HM^d(E\cap
		\ann(x,r,\rho))+\HM(E^{rec}\cap \cball(x,r))\big).
	\]
	Let $\varphi:\mathbb{R}^n\to \mathbb{R}^n$ be the mapping given by 
	\[
		\varphi(y)=\begin{cases}
			y,& y\in \mathbb{R}^n\setminus \cup_{i=1}^m \cball_i\\
			\phi_{x_i}(y),& y\in \cball_i.
		\end{cases}
	\]
	Then $\varphi$ is Lipschitz, thus for any $0<\delta<1/2$, setting
	$\cball_i=\cball(x_i,\rho_i)$, $\oball_i=\oball(x_i,\rho_i)$ and
	$\rho_i'=(1-\delta)\rho_i$, we get that
	\[
		\begin{aligned}
			\HM^d(E\cap \cup_{i=1}^m\oball_i)&\leq M(\diam O)\HM^d(\varphi(E\cap \cup_{i=1}^m
			\oball_i))+\varepsilon(\diam O)\\
			&\leq  c_0M(\diam O)\sum_{i=1}^m\left(\HM^d(E\cap
			\ann(x_i,\rho_i',\rho_i))+\HM^d(E^{rec}\cap
			\cball(x_i,\rho_i'))\right)+\varepsilon(\diam O).\\
		\end{aligned}
	\]
	By the arbitrariness of $\delta$, we get that 
	\[
			\HM^d(E\cap \cup_{i=1}^m\oball_i)\leq c_0M(\diam O)
			\sum_{i=1}^m\HM^d(E^{rec}\cap \oball_i)+\varepsilon(\diam O).
	\]
	Since $\HM^d(E^{rec}\cap \oball_i)\leq \tau \omega_d \rho_i^d\leq 2^{d+1}
	\tau\HM^d(E^{irr}\cap \oball_i)$, we get that
		\[
			\sum_{i=1}^m\HM^d(E^{irr}\cap \oball_i)\leq 
			\HM^d(E\cap \cup_{i=1}^m\oball_i)\leq c_0M(\diam O) 
			\sum_{i=1}^m2^{d+1}\tau\HM^d(E^{irr}\cap
			\oball_i)+\varepsilon(\diam O),
	\]
	thus
	\[
		\left(1-2^{d+1}c_0M(\diam O)\tau\right)\sum_{i=1}^m\HM^d(E^{irr}\cap
		\oball_i) \leq \varepsilon(\diam O).
	\]
	By the arbitrariness of $\tau$, we get that 
	\[
		\HM^d(E^{irr}\cap O)\leq \varepsilon(\diam O).
	\]
\end{proof}
\begin{corollary}
	\label{co:rect}
	Suppose that $E\in \QM_d(\mathbb{R}^n,U,M,\varepsilon)$.
	If $\limsup_{r\to 0+}r^{-d}\varepsilon(r)<2^{-d}\omega_d$,
	then $E\cap U$ is $d$-rectifiable.
\end{corollary}
\begin{proof}
	Write $E\cap U=E^{rec}\cup E^{irr}$. Since $\limsup_{r\to
	0+}r^{-d}\varepsilon(r)<2^{-d}\omega_d$, by Lemma
	\ref{le:irr}, we get that $\Theta^{\ast d}(E^{irr},x)<2^{-d}$ for every
	point $x\in U$. Thus $\HM^d(E^{irr})=0$.
\end{proof}
\begin{lemma} \label{le:uAR}
	Suppose that $E\in \QM_d(\mathbb{R}^n,U,M,\varepsilon)$. Let $c_1=c_1(n,d)>0$ 
	be the constant in Lemma \ref{le:ffpa}, let $\eta>1$ and
	$r_0>0$ be such that $2(n-d)c_0 (\ln \eta)^{-1/2}M(\eta r_0)\leq 1$.
	Then for any $x\in E\cap U$ and $0<r<\min\{r_0,\eta^{-1}
	\dist(x,\mathbb{R}^n\setminus U)\}$, we have that
	\[
		\HM^d(E\cap \cball(x,r))\leq 2c_1M(2\eta r)(\eta r)^d + \varepsilon(2\eta r).
	\]
\end{lemma}
\begin{proof}
	We take $\alpha_1=1$, and 
	\[
		\gamma=\left(1-\frac{1}{c_0 M(2 \eta
		r)}\right)^{\frac{1}{2(n-d)}},\ 
		\alpha_{m+1}=\Big( 1- \gamma^m \Big)^{-1} \alpha_m, \ m\geq 1.
	\]
	Since $(1+x)^{\beta}\leq 1+\beta x$ for $\beta\in [0,1]$ and $x\geq -1$, we
	have that 
	\[
		\gamma\leq \left( 1- 2(n-d)(\ln \eta)^{-1/2}\right)^{1/(2(n-d))}\leq
		1-(\ln \eta)^{-1/2}
	\]

	Since $x/(1+x)\leq \ln (1+x)\leq x$ for $-1<x<1$, we have that 
	\[
		\ln \alpha_{m+1}=-\sum_{k=1}^{m}\ln \Big( 1-\gamma^k \Big)\leq
		\sum_{k=1}^{m}\frac{\gamma^k}{1-\gamma^k}\leq
		\sum_{k=1}^{m}\frac{\gamma^k}{1-\gamma}\leq
		\frac{\gamma}{(1-\gamma)^2}< \ln \eta.
	\]
	Thus $\alpha_m< \eta$ for any $m\geq 1$. Put $r_1=r$ and $r_m=\alpha_m r$.
	By Lemma \ref{le:ffpa}, we can find Lipschitz mappings $\phi_m:\mathbb{R}^n\to
	\mathbb{R}^n$ such that $\phi_m(\cball(x,r_{m+1}))\subseteq
	\cball(x,r_{m+1})$, $\phi_m\vert_{\mathbb{R}^n\setminus \cball(x,r_{m+1})}=
	\id_{\mathbb{R}^n\setminus \cball(x,r_{m+1})}$ and 
	\[
		\HM^d(\phi_m(E\cap \oball(x,r_{m+1})))\leq c_0\HM^d(E\cap
		\ann(x,r_{m+1}, r_m))+ c_1 (\alpha_{m+1}-\alpha_{m})^{d-n}(\alpha_{m+1} )^n
		r^d.
	\]
	Thus, setting $\oball_m = \oball(x,r_m)$, 
	\[
		\begin{aligned}
			\HM^d(E\cap \oball_{m+1})&\leq M(2 \eta r)\HM^d(\phi_m(E\cap
			\oball_{m+1}))+\varepsilon(2 \eta r)\\
			&\leq c_0M(2\eta r) \HM^d(E\cap \ann(x,r_{m+1},r_m))+
			c_1M(2\eta r)(\alpha_{m+1}-\alpha_{m})^{d-n}\alpha_{m+1}^n r^d
			+\varepsilon(2\eta r),
		\end{aligned}
	\]
	and we get that 
	\[
		\HM^d(E\cap \cball(x,r_m))\leq \frac{c_0
		M(2\eta r)-1}{c_0M(2\eta r)}\HM^d(E\cap \oball_{m+1})+
		\frac{c_1M(2\eta r)(\alpha_{m+1}-\alpha_{m})^{d-n}\alpha_{m+1}^n r^d
		+\varepsilon(2\eta r)}{c_0M(2\eta r)},
	\]
	thus
	\[
		\HM^d(E\cap \cball(x,r))\leq \left(\frac{c_0 M(2\eta r)-1}{c_0M(2 \eta r)}
		\right)^{m}\HM^d(E\cap \oball_{m+1})+
		\sum\frac{c_1c_0^{-1}r^d \alpha_{k+1}^n}{(\alpha_{k+1}-\alpha_k)^{n-d}}
		\left(\frac{c_0M(2\eta r)-1}{c_0M(2\eta r)}\right)^k+ \varepsilon(2\eta r).
	\]
	Since $\HM^d(E\cap \cball(x,\eta r))<\infty$, we get that
	\[
		\begin{aligned}
			\HM^d(E\cap \cball(x,r))&\leq c_1c_0^{-1}
			\sum_{k=1}^{\infty}\frac{\alpha_{k+1}^n}{(\alpha_{k+1}-\alpha_k)^{n-d}}
			\left(\frac{c_0M(2\eta r)-1}{c_0M(2\eta r)}\right)^{k} r^d+
			\varepsilon(2\eta r)\\
			&\leq c_1c_0^{-1} \eta^d \sum_{k=1}^{\infty}(1-\alpha_k/
			\alpha_{k+1})^{d-n}\gamma^{2(n-d)k} r^d+\varepsilon(2 \eta r)\\
			&= c_1c_0^{-1} \eta^d\sum_{k=1}^{\infty}\gamma^{(n-d)k} r^d+
			\varepsilon(2\eta r)= c_1c_0^{-1} \eta^d
			\frac{\gamma^{n-d}}{1-\gamma^{n-d}}r^d+ \varepsilon(2\eta r)\\
			&\leq 2c_1\eta^dM(2\eta r)r^d+ \varepsilon(2\eta r).
		\end{aligned}
	\]
\end{proof}

\begin{lemma}
	\label{le:density}
	Suppose that $E_k\in \QM_d(\mathbb{R}^n,U,M_k,\varepsilon_k)$,
	$M(r)=\limsup_{k}M_k(r)$, $\varepsilon(r)=\limsup_{k}\varepsilon_k(r)$,
	$\HM^d \mr (E_k\cap U)\wc \mu$, $M(0+)<\infty$ and $\varepsilon_0=\limsup_{r\to
	0}r^{-d}\varepsilon(r)<\infty$. Then for any $x\in U\cap \spt \mu$, 
	\[
		\Theta^{\ast d}(\mu,x)\leq \left(2 c_1
		M(0+)\omega_d^{-1}+2^d \omega_d^{-1}\varepsilon_0\right)\exp
		\left(4d(n-d)^2c_0^2 M(0+)^2\right).
	\]
\end{lemma}

\begin{proof}
	Take $\eta>\exp \left( 4(n-d)^2c_0^2 M(0+)\right)$. Then $2(n-d)c_0 (\ln
	\eta)^{-1/2}M(0+)< 1$, and there is a radius $r_0>0$ such that $2(n-d)c_0
	(\ln \eta)^{-1/2}M(\eta r_0)< 1$, thus $2(n-d)c_0 (\ln \eta)^{-1/2}
	M_k(\eta r_0)< 1$ for $k$ large enough. For any $x\in U\cap \spt \mu$ and 
	$0<r<\min\{r_0,\eta^{-1}\dist(x,\mathbb{R}^n\setminus U)\}$,  we have that 
	\[
		\HM^d(E_k\cap \cball(x,r))\leq 2c_1M_k(2\eta r)(\eta
		r)^d+\varepsilon_k(2\eta r),
	\]
	thus
	\begin{equation}\label{eq:ual}
		\mu(\oball(x,r))\leq \liminf_{k\to \infty}\HM^d(E_k\cap \oball(x,r))\leq
		2c_1 M(2\eta r)(\eta r)^d+ \varepsilon(2\eta r),
	\end{equation}
	and 
	\[
		\Theta^{\ast d}(\mu,x)\leq 2c_1 M(0+)\eta^d \omega_d^{-1}+ 2^d \eta^d
		\omega_d^{-1}\limsup_{r\to 0}r^{-d}\varepsilon(r)\leq
		\left(2c_1M(0+)\omega_d^{-1}+2^d \omega_d^{-1} \varepsilon_0\right)\eta^d.
	\]
\end{proof}
\begin{lemma}
	\label{le:bl}
	There exist constants $c_3=c_3(n,d) >0$ and $c_4=c_4(n)>0$ such that
	for any $M$, $\varepsilon$, $E\in \QM_d(\mathbb{R}^n,U, M,
	\varepsilon)$, open set $O\subseteq U$,
	$\delta>0$ and $0<\tau<\min\{(10\sqrt{n})^{-d},
	c_3^{-1}M(\diam O)^{-1}\}$, by setting
	\[
		m(O,x,\delta)=\inf\left\{r^{-d}\HM^d(E\cap
		\cball(x,r)):0<r<\min\{\dist(x,\mathbb{R}^n\setminus O),\delta\}\right\}
	\]
	and
	\[
		E(O,\delta,\tau)=\{x\in E\cap O :m(O,x,\delta)\leq \tau\},
	\]
	we have that  
	\[
		\HM^d(E(O,\delta,\tau))\leq \frac{c_4 \varepsilon(\diam O)}
		{1-c_3 M(\diam O )\tau}.
	\]
\end{lemma}
\begin{proof}
	For any $x\in E(O,\delta,\tau)$, we denote $m_x=m(O,x,\delta)$ and choose 
	radius $r_x>0$ such that
	$r_x<\min\{\dist(x,\mathbb{R}^n\setminus O),\delta\}$,
	$\HM^d(E\cap \partial \cball(x,r_x))=0$ and 
	\[
		\frac{\HM^d(E\cap \cball(x,r_x))}{r_x^d}<\left(1+\left(6c_0\sqrt{n}\cdot
		m_x\right)^{1/d}\right) m_x.
	\]
	Applying Besicovitch's covering theorem to $\{\cball(x,r_x)\}$, we can
	find constant $c_4=c_4(n)$ of positive integers and balls 
	$\{\cball_{i,j}\}_{j\in J_i}$, $1\leq i\leq c_4$, such that
	$\cball_{i,j_1}\cap\cball_{i,j_2}=\emptyset$ for any $j_1\neq j_2$, and
	\[
		E(O,\delta,\tau)\subseteq \bigcup_{i=1}^{c_4}\bigcup_{j\in J_i} \cball_{i,j}.
	\]
	By Lemma \ref{le:ffpa}, we can find Lipschitz mappings
	$\varphi_i:\mathbb{R}^n\to \mathbb{R}^n$ such that
	$\varphi_i(\cball_{i,j})\subseteq \cball_{i,j}$ for $j\in j_i$, 
	$\varphi_i(y)=y$ for $y\notin \cup_{j\in J_i}\cball_{i,j}$, and 
	\[
		\HM^d(E\cap \cup_{j\in J_i}\cball_{i,j})\leq c_0M(\diam O)\HM^d(E\cap
		\cup_{j\in J_i}\ann_{i,j}) +\varepsilon(\diam O),
	\]
	where $\ann_{i,j}=\ann(x_{i,j},(1-\delta_{i,j})r_{i,j},r_{i,j})$ and
	$\delta_{i,j}=(6c_0\sqrt{n}\cdot m_{i,j})^{1/d}$. Since $m_{i,j}<\tau$ and
	\[
		\HM^d(E\cap \ann_{i,j})\leq \left(1+\left(6c_0\sqrt{n}\cdot
		m_x\right)^{1/d}\right)
		m_{i,j}r_{i,j}^d-m_{i,j}(1-\delta_{i,j})^dr_{i,j}^{d}\leq
		(d+1)\delta_{i,j}m_{i,j}r_{i,j}^{d},
	\]
	we get that 
	\[
		\begin{aligned}
			\HM^d(E\cap \cup_{j\in J_i}\cball_{i,j})&\leq c_0M(\diam
			O)(d+1)(6c_0\sqrt{n})^{1/d}\tau\sum_{j\in J_i}m_{i,j}r_{i,j}^d+
			\varepsilon(\diam O)\\
			&\leq c_0M(\diam
			O)(d+1)(6c_0\sqrt{n})^{1/d}\tau \HM^d(E\cap \cup_{j\in
			J_i}\cball_{i,j})+\varepsilon(\diam O).
		\end{aligned}
	\]
	We take $c_3=(d+1)\sqrt{n}^{d/2}c_0^{(d+1)/d}$. Then
	\[
		\HM^d(E\cap \cup_{j\in J_i}\cball_{i,j})\leq \frac{\varepsilon(\diam
		O)}{1-c_3M(\diam O)\tau}.
	\]
	Hence 
	\[
		\HM^d(E(O,\delta,\tau))\leq \sum_{i=1}^{c_4}\HM^d(E\cap \cup_{j\in
		J_i}\cball_{i,j})\leq \frac{c_4\varepsilon(\diam
		O)}{1-c_3M(\diam O)\tau}.
	\]
\end{proof}

\begin{lemma}
	\label{le:lden}
	Suppose that $E_k\in \QM_d(\mathbb{R}^n,U,M_k,\varepsilon_k)$, $M(r)=\limsup_{k}M_k(r)$, $\varepsilon(r)=\limsup_k
	\varepsilon_k(r)$, $\HM^d\mr (E_k\cap U) \wc \mu$, $M(0+)<\infty$ and $\varepsilon_0=
	\limsup_{r\to 0}r^{-d}\varepsilon(r)<\infty$. Then for any $x\in U\cap \spt
	\mu$, if $\Theta^{\ast d}(\mu,x)>6^{d+1}c_4 \omega_d^{-1}\varepsilon_0$, then
	$\Theta_{\ast}^d(\mu,x)\geq (2 \omega_d c_0M(0+))^{-1}$. In particular, if
	$\varepsilon_0=0$, then $\Theta_{\ast}^d(\mu,x)\geq
	(2 \omega_d c_0M(0+))^{-1}$ for $\HM^d$-a.e. $x\in U\cap \spt \mu$. 
\end{lemma}
\begin{proof}
	We take a sequence $\{r_m\}$ of decreasing positive numbers such that $r_m\to 0$ and
	\[
		\lim_{m\to \infty}\frac{\mu(\oball(x,r_m))}{\omega_dr_m^{d}}=\Theta^{\ast
		d}(\mu,x)>2^{d+1}c_4 \varepsilon_0,
	\]
	thus
	\[
		\liminf_{m\to \infty}\liminf_{k\to \infty}\frac{\HM^d(E_k\cap
		\oball(x,r_m))}{\omega_dr_m^{d}}\geq \Theta^{\ast
		d}(\mu,x)>2^{d+1}c_4 \varepsilon_0.
	\]
	For any $0<\tau<1/(2c_0M(0+))$, we take $0<r_0<\frac{1}{2}
	\dist(x,\mathbb{R}^n\setminus U)$ such that $2c_0M(2r_0)<1$.
	Since $E_k(U,2r_m,\tau)\cap \oball(x,r_m)=E_k(\oball(x,3r_m),2r_m,\tau)\cap
	\oball(x,r_m)\subseteq E_k(\oball(x,3r_m),r_m,\tau)$ when $m$ large enough
	such that $r_m<r_0/2$, by Lemma \ref{le:bl}, we get that 
	\[
		\limsup_{k\to \infty}\HM^d(E_k(U,2r_m,\tau)\cap \oball(x,r_m))\leq
		2c_4\varepsilon(6r_m)
	\]
	and 
	\[
		\limsup_{m\to \infty}\limsup_{k\to \infty}\frac{1}{r_m^d}
		\HM^d\big(E_k(U,2r_m,\tau)\cap \oball(x,r_m)\big)\leq 2c_46^d \varepsilon_0.
	\]
	We get that 
	\[
		\liminf_{m\to \infty}\liminf_{k\to \infty}\frac{\HM^d(E_k\cap
		\oball(x,r_m)\setminus E_k(U,2r_m,\tau))}{r_m^{d}}\geq \omega_d
		\Theta^{\ast d}(\mu,x)-2c_46^d \varepsilon_0>0.
	\]
	Since $r_m>r_{m+1}$, we get that $E_k(U,r_m,\tau)\subseteq
	E(U,r_{m+1},\tau)$. Taking $m_0>0$ such that $r_0'=r_{m_0}<r_0/2$,
	we have that for any $m\geq m_0$, there is an integer $k_{m}>0$ such that 
	$E_k\cap \oball(x,r_m)\setminus E_k(U,r_0',\tau)\neq \emptyset$ for 
	$k\geq k_m$. We take $y_{m,k}\in E_k\cap \oball(x,r_m)\setminus E_k
	(U,r_0',\tau)$, then for any $0<r<r_0'$,
	\[
		\HM^d(E_k\cap \cball(x,r))\geq \HM^d(E_k\cap
		\cball(y_{m,k},r-|x-y_{m,k}|)) \geq \HM^d(E_k\cap
		\cball(y_{m,k},r-r_m))\geq \tau (r-r_m)^d.
	\]
	Hence 
	\[
		\mu(\cball(x,r))\geq \limsup_{k\to \infty}\HM^d(E_k\cap \cball(x,r))\geq \tau
		(r-r_m)^d.
	\]
	Since $r_m\to 0$, we get that $\mu(\cball(x,r))\geq \tau r^d$, and
	$\Theta_{\ast}^d(\mu,x)\geq \omega_d^{-1}\tau$.
\end{proof}

\begin{proposition}
	\label{prop:limitradon}
	Suppose that $E_k\in \QM_d(\mathbb{R}^n,U,M_k,\varepsilon_k)$, $M(r)=\limsup_{k}M_k(r)$,
	$\varepsilon(r)=\limsup_k\varepsilon_k(r)$ and $\HM^d\mr (E_k\cap U) \wc \mu$. 
	If $M(0+)<\infty$ and $\lim_{r\to 0}r^{-d}\varepsilon(r)=0$, then $\mu \mr U$ is $d$-rectifiable. 
\end{proposition}

\begin{proof}
	Write $E_k\cap U=E_{k}^{rec}\sqcup E_{k}^{irr}$, $E_k^{rec}$ is
	$d$-rectifiable and $\HM^d$-measurable, $E_k^{irr}$ is purely
	$d$-unrectifiable and $\HM^d$-measurable. Put $\mu_k=\HM^d\mr E_{k}^{rec}$. 
	By Lemma \ref{le:irr}, we have that $\HM^d\mr E_k^{irr}\wc 0$ and $\mu_k\wc
	\mu$. Put $E=U\cap \spt \mu$. By Lemma \ref{le:density} and Lemma
	\ref{le:lden}, there is a constant $c_2>0$ which only depends on $n,d,M(0+)$ 
	such that $1/c_2\leq\Theta_{\ast}^d(\mu,x)\leq \Theta^{\ast d}(\mu,x)\leq 
	c_2$ for $\HM^d$-a.e. $x\in E$, thus there is a Borel function $\theta:E\to
	\mathbb{R}$ such that $\mu=\theta \HM^d\mr E$. We assume by contradiction 
	that $E\cap U$ is not $d$-rectifiable, and $E\cap U=E^{rec}\sqcup E^{irr}$. 
	Then there exists $x\in E^{irr}$ such that 
	\[
		\Theta^d(E^{rec},x)=0,\ 2^{-d}<\Theta^{\ast d}(E^{irr},x)\leq 1,
		\text{ and }1/c_2\leq \Theta^d(\mu,x)\leq c_2.
	\]
	For any $0<\delta<\min\{1/10,(c_0 (c_2)^2 2^{d+4}\sqrt{n} M(0+))^{-1}\}$ and $0<\tau<(10\sqrt{n})^{d-n}\delta^d$, 
	we take radius $\rho>0$ such that
	$\rho<\frac{1}{2}\dist(x,\mathbb{R}^n\setminus U)$,  
	\[
		\frac{\HM^d(E^{rec}\cap \cball(x,r))}{\omega_d r^d}\leq \tau
		\text{ and }\frac{\HM^d(E\cap \cball(x,r))}{\omega_d r^d}\leq
		(1+\tau)\Theta^{\ast d}(E,x) \text{ for any } 0<r\leq
		2\rho,
	\]
	and 
	\[
		\frac{\HM^d(E^{irr}\cap \cball(x,\rho))}{\omega_d \rho^d}\geq 
		(1-2\tau)\Theta^{\ast d}(E,x). 
	\]
	For any $i_1, \cdots, i_n\in \mathbb{Z}$, we denote by
	$C_{i_1,\cdots,i_n}$ the cube given by $(\delta \rho i_1,\cdots,\delta \rho
	i_n)+[0,\delta \rho]^n$. Let $\mathscr{F}$ be the collection of all cubes
	$C_{i_1,\cdots,i_n}$ which is contained in $\oball(x,(1+3\sqrt{n}\delta)\rho)
	$. Put $\mathscr{F}'=\{\Delta\in \mathscr{F}: \Delta \subseteq \intr(\cup
	\mathscr{F}) \}$. By Theorem \ref{thm:ffp}, there exist a Lipschitz mapping
	$\phi:\mathbb{R}^n\to \mathbb{R}^n$, an upper semi-continuous function
	$\lambda:\mathbb{R}^n\to [0,\infty)$ and open set $W\subseteq \mathbb{R}^n$ 
	such that $\phi(x)=x$ for $x\in \mathbb{R}^n\setminus 
	\cup \mathscr{F}$,  $\phi(\Delta)\subseteq \Delta$ for any 
	$\Delta\in \mathscr{F}$, $E\subseteq W$, $\phi(W\cap
	\cball(x,\rho))\subseteq \cup \mathscr{F}_d'$, 
	$\lambda(x)=1$ for $x\in \mathbb{R}^n\setminus \cup \mathscr{F}$,
	for any $d$-set $Z\subseteq W$,
	\[
		\HM^d(\phi(Z))\leq \int_{x\in Z}\lambda(x)\ud \HM^d(x),
	\]
	and for any $\mathscr{K} \subseteq \mathscr{F}$
	\[
		\HM^d(\phi(E\cap (\cup \mathscr{K})))\leq \int_{E\cap (\cup
		\mathscr{K})}\lambda(x) \ud \HM^d(x)\leq c_0\HM^d(E\cap (\cup 
		\mathscr{K})).
	\]

	Since $\HM^d\mr (E_k\cap U)\wc \mu$, $E=U\cap \spt \mu$, $E\subseteq W$ and
	$W$ is open, we see that 
	\[
		\lim_{k\to \infty}\HM^d(E_k\cap U\setminus W)=0,
	\]
	and 
	\[
		\lim_{k\to \infty}\HM^d(\phi(E_k\cap U\setminus W))=0.
	\]
	Setting $D=\cup \mathscr{F}$, $\cball_{\rho}'=\cball(x,(1+3\sqrt{n}\delta)
	\rho)$ and $\cball_{\rho}=\cball(x,\rho)$, then by \eqref{eq:ffp4}, we have
	that
	\[
		\HM^d(\intr(D)\cap \phi(W))=0,
	\]
	thus $\phi(E_k\cap \cball_{\rho}\cap W)\subseteq \intr(D)\cap \phi(W)$, and
	\[
		\HM^d(\phi(E_k\cap \cball_{\rho}\cap W))=0,\ \lim_{k\to\infty}
		\HM^d(\phi(E_k\cap \cball_{\rho}'\setminus W))=0,
	\]
	\[
		\HM^d(\phi(E_k\cap \cball_{\rho}'))\leq \HM^d(\phi(E_k\cap
		\cball_{\rho}'\setminus W))+\int_{E_k\cap W\cap
		\cball_{\rho}'\setminus \cball_{\rho}}\lambda(x)\ud \HM^d(x).
	\]
	Thus, setting $\rho_1=(1+3\sqrt{n}\delta)\rho$ and
	$\oball_{\rho}=\oball(x,\rho)$, we have that
	\[
		\begin{aligned}
			\limsup_{k\to \infty}\HM^d \left( \phi (E_k\cap \cball_{\rho}')\right)&\leq
			\limsup_{k\to \infty}\int_{W\cap \cball_{\rho}'\setminus \cball_{\rho}}
			\lambda\ud \HM^d\mr E_k \leq \int_{W\cap \cball_{\rho}'\setminus 
			\oball_{\rho}} \lambda\ud \mu\\
			&\leq c_2\int_{W\cap \cball_{\rho}'\setminus \oball_{\rho}} \lambda\ud
			\HM^d\mr E
			\leq c_2c_0\HM^d(E\cap W\cap \cball_{\rho}'\setminus \oball_{\rho})\\
			&\leq c_2c_0\HM^d(E\cap \cball_{\rho}'\setminus \oball_{\rho})\\
			&\leq
			c_2c_0\left((1+\tau)\rho_1^d-(1-2\tau)\rho^d\right)\Theta^{\ast
			d}(E,x)\omega_d \\
			&\leq c_2c_0\left((2+2^d)\tau+2^{d+1}\sqrt{n} \delta\right)\omega_d \rho^d.
		\end{aligned}
	\]
	We have that 
	\[
		\begin{aligned}
			\mu(\oball(x,\rho_1))&\leq \liminf_{k\to \infty}\HM^d \left(
			E_k\cap \oball(x,\rho_1)\right)\\
			&\leq \liminf_{k\to \infty}\left( M(2\rho_1)\HM^d \left( \phi
			(E_k\cap \cball(x,\rho_1))\right) +\varepsilon(2\rho_1)\right)\\
			&\leq M(2\rho_1)c_2c_02^{d+1}(\tau+\sqrt{n}\delta)\omega_d \rho^d+
			\varepsilon(2 \rho_1).
		\end{aligned}
	\]
	Thus 
	\[
		\Theta^{\ast d}(\mu,x)\leq	 M(0+)c_2c_0 2^{d+1}(\tau+\sqrt{n}\delta)
		<\frac{1}{2 c_2},
	\]
	and this is a contradiction since $\Theta_{\ast}^d(\mu,x)\geq 1/c_2$.

\end{proof}
The following corollary is a direct consequence of Lemma \ref{le:density}, 
Lemma \ref{le:lden} and Proposition \ref{prop:limitradon}.
\begin{corollary}
	\label{cor:rlr}
	Suppose that $E_k\in \QM_d(\mathbb{R}^n,U,M_k,\varepsilon_k)$,
	$M(r)=\limsup_{k}M_k(r)$,
	$\varepsilon(r)=\limsup_k\varepsilon_k(r)$ and $\HM^d\mr (E_k\cap U) \wc \mu$. 
	If $M(0+)<\infty$ and $\lim_{r\to 0}r^{-d}\varepsilon(r)=0$, then $U\cap
	\spt \mu$ is $d$-rectifiable and there is a constant $c_2=c_2(n,d,M(0+))>0$
	such that for $\HM^d$-a.e. $x\in U\cap \spt\mu$, 
	\[
		c_2^{-1}\leq \Theta^d(\mu,x)\leq c_2.
	\]
\end{corollary}

\begin{lemma} \label{le:mp}
	Let $A\subseteq \oball(x,r)$ be a convex set in $\mathbb{R}^{n}$. Let $P$ be
	a $d$-plane through $x$. Suppose that $0<\varepsilon, \delta<1/2$,
	$A\subseteq  (P+\cball(0,\varepsilon r))\cap \cball(x,(1-\delta)r)$. Then 
	there is a Lipschitz mapping $\varphi:\mathbb{R}^n\to \mathbb{R}^n$ such that 
	\[
		\varphi\vert_{\mathbb{R}^n\setminus
		\oball(x,r)}=\id_{\mathbb{R}^n\setminus \oball(x,r)},\
		\varphi\vert_{A}=P_{\natural}\vert_A, \varphi(\oball(x,r))\subseteq
		\oball(x,r)\text{ and }\Lip(\varphi)\leq 3+\varepsilon/\delta.
	\]
\end{lemma}
The proof can be found in \cite[Lemma 2.5]{FFL:2019}, so we omit it here.
\begin{lemma}\label{le:qmtpro}
	Suppose that $E_k\in \QM_d(\mathbb{R}^n,U,M_k,\varepsilon_k)$,
	$M(r)=\limsup_{k}M_k(r)$, $\varepsilon(r)=\limsup_{k}\varepsilon_k(r)$ 
	$\HM^d \mr (E_k\cap U)\wc \mu$, $M(0+)<\infty$ and 
	$\lim_{r\to 0}r^{-d}\varepsilon(r)=0$. For any $x\in U\cap \spt \mu$, if
	both of $\Theta^d(\mu,x)$ and $T=
	\Tan^d(\mu,x)$ exist, then for any decreasing sequence $\{r_m\}$ of
	numbers with $r_m\to 0$, setting $\cball_m=\cball(x,r_m)$, we have that 
	\begin{equation}\label{eq:qmtpro1}
		\limsup_{m\to \infty}\limsup_{k\to \infty}\frac{1}{r_m^d}
		\left(\HM^d(E_k^{rec}\cap \cball_m)-M(0+)\HM^d\left(T_{\natural}(E_k^{rec}\cap
		\cball_m)\right)\right)\leq 0,
	\end{equation}
	where $E_k^{rec}$ is the $d$-rectifiable part of $E_k\cap U$.
\end{lemma}
\begin{proof}
	Without loss of generality, we assume that $x=0\in U\cap \spt \mu$. Write 
	$E_k\cap U=E_{k}^{rec}\sqcup E_k^{irr}$. By Lemma 
	\ref{le:irr}, we see that $\HM^d\mr E_{k}^{irr}\wc 0$, thus $\HM^d \mr E_k^{rec}
	\wc \mu$. 	For any $0<\delta<1/10$, taking $\tau=\delta^2$, we see that 
	\[
		\lim_{m\to \infty}\frac{\mu(\cball(0,r_m))}{\omega_d r_m^d}=
		\Theta^d(\mu,0)\in [1/c_2,c_2] \text{
		and } \lim_{m\to \infty}\frac{\mu(\cball(0,r_m)\setminus
		\mathscr{C}(T,0,\tau))}{r_m^d}=0.
	\]

	Put $A_m=(T+\cball(0,\tau r_m))\cap \cball(0,(1-2\delta)r_m)$. Then we see
	that $A_m\subseteq T+\cball(0,\tau r_m)$ and $A_m+\cball(0,\delta r_m)\subseteq
	\oball(0,r_m)$. By Lemma \ref{le:mp}, there is a Lipschitz mapping
	$\varphi_m:\mathbb{R}^n\to \mathbb{R}^n$ such that
	$\varphi_m(\cball(0,r_m))\subseteq \cball(0,r_m)$,
	$\varphi_m\vert_{\mathbb{R}^n\setminus
	\oball(0,r_m)}=\id_{\mathbb{R}^n\setminus \oball(0,r_m)}$,
	$\varphi_m\vert_{A_m}=T_{\natural}\vert_{A_m}$ and $\Lip(\varphi_m)\leq
	3+\tau/\delta$. Setting $r_m'=(1-2 \delta)r_m$,
	$\cball_m'=\cball(0,r_m')$ and $\oball_m'=\oball(0,r_m')$, we get that 
	\[
		\begin{aligned}
			\HM^d(\varphi_m(E_k\cap \cball_m))&\leq \Lip(\varphi_m)^d\HM^d(E_k\cap
			\cball_m\setminus A_m)+\HM^d(T_{\natural}(E_k\cap A_m))\\
			&\leq 4^d \left( \HM^d(E_k\cap \cball_m\setminus
			\cball_m')+\HM^d(E_k\cap \cball_m\setminus
			\mathscr{C}(T,0,\tau)\right)+\HM^d(T_{\natural}(E_k\cap A_m)) .
		\end{aligned}
	\]
	Since $\HM^d(E_k\cap \cball(0,r_m))\leq M_k(2r_m)\HM^d(\varphi_m(E_k\cap
	\cball(0,r_m)))+ \varepsilon_k(2r_m)$, we get that 
	\begin{multline*}
		\HM^d(E_k\cap \cball_m)-M_k(2r_m)\HM^d(T_{\natural}(E_k\cap \cball_m))\\
		\leq 4^d M_k(2r_m)\left( \HM^d(E_k\cap \cball_m\setminus
		\cball_m')+\HM^d(E_k\cap \cball_m\setminus
		\mathscr{C}(T,0,\tau)\right) +\varepsilon_k(2r_m).
	\end{multline*}
	By Lemma \ref{le:irr}, we have that $\HM^d(E_{k}^{irr}\cap \cball_m)\leq
	\varepsilon_m(2r_m+)$, thus
	\begin{multline*}
		\limsup_{k\to \infty}\Big(\HM^d(E_k^{rec}\cap
		\cball_m)-M(2r_m)\HM^d(T_{\natural}(E_k^{rec}\cap \cball_m))\Big)\\
		\leq 4^d
		M(2r_m)\big(\mu(\cball_m\setminus \oball_m')+ \mu(\cball_m\setminus
		\mathscr{C}(T,0,\tau))\big)+2\varepsilon(2r_m+),
	\end{multline*}
	and 
\begin{multline*}
		\limsup_{m\to \infty}\limsup_{k\to \infty}\frac{1}{r_m^d}\Big(\HM^d(E_k^{rec}\cap
		\cball_m)-M(2r_m)\HM^d(T_{\natural}(E_k^{rec}\cap \cball_m))\Big)\\
		\leq M(0+)4^d(1-(1-2 \delta)^d)\omega_d \Theta^d(\mu,0).
\end{multline*}
	Let $\delta$ tend to 0, we get that 
	\[
		\limsup_{m\to \infty}\limsup_{k\to \infty}\frac{1}{r_m^d}\Big(\HM^d(E_k^{rec}\cap
		\cball_m)-M(2r_m)\HM^d(T_{\natural}(E_k^{rec}\cap \cball_m))\Big)\leq 0.
	\]
	Since $M(2r_m)\to M(0+)$, we get that \eqref{eq:qmtpro1} holds.
\end{proof}
\begin{lemma}
	\label{le:lpd}
	Suppose that $E_k\in \QM_d(\Omega,U,M_k,\varepsilon_k)$,
	$M(r)=\limsup_{k}M_k(r)$, $\varepsilon(r)=\limsup_{k}\varepsilon_k(r)$,
	$M(0+)<\infty$, $\lim_{r\to 0+}r^{-d}\varepsilon(r)=0$ and $\HM^d\mr (E_k\cap
	U)\wc \mu$. For any $x\in U\cap \spt \mu$, if $\Theta^d(\mu,x)\in 
	(0,\infty)$ and $T=\Tan^d(\mu,x)$ exists, then
	\begin{equation}\label{eq:lpd0}
		\liminf_{r\to 0}\liminf_{k\to \infty}\frac{\HM^d(T_{\natural}(E_k\cap
		\cball(x,r)))}{\omega_d r^d}\geq 1.
	\end{equation}

\end{lemma}
\begin{proof}
	For any $0<\delta<\min\{(10\sqrt{n})^{-1},(5^dc_0d M(0+))^{-1}\}$, taking $\tau=\delta^2$, setting
	$r'=(1-\delta)r$, $\cball_r=\cball(x,r)$, $\oball_r=\oball(x,r)$,
	$\cball_r'=\cball(x,r')$ and $A_r=\cball_r\cap (T+\cball(0,\tau r))$,
	we see that 
	\[
		\lim_{r\to 0}\frac{\mu(\cball(x,r))}{\omega_d r^d}=\Theta^d(\mu,x)
		\text{ and }\lim_{r\to 0}\frac{\mu(\cball(0x,r\setminus
		\mathscr{C}(T,x,\tau))}{r^d}=0.
	\]
	Let $\mathscr{F}^r$ be the collection of cubes $C_{i_1,\cdots,i_n}=(\tau r
	i_1,\cdots,\tau r i_n)+[0,\tau r]^n$ which is contained in
	$\oball(x,r)\setminus A_r$. We have that 
	\[
		\lim_{r\to 0}\frac{1}{r^d}\mu(\cup \mathscr{F}^r)=0.
	\]
	There exists $r_{\tau}>0$ such that $\mu(\cup \mathscr{F}^r)< c_0^{-1}
	(\tau r)^d$ for any $0<r\leq r_{\tau}$, thus there exit $k_{\tau,r}$ such that
	$\HM^d(E_k\cap \cup \mathscr{F}^r)<c_0^{-1}(\tau r)^d$ for any $k\geq k_{\tau,r}$. 
	By Theorem \ref{thm:ffp}, there 
	exist $\phi_{r,k}:\mathbb{R}^n\to \mathbb{R}^n$ such that $\phi_{r,k}(\cup
	\mathscr{F}^r) \subseteq \cup \mathscr{F}^r$,
	$\phi_{r,k}\vert_{\mathbb{R}^n\setminus \cup \mathscr{F}^r}=
	\id_{\mathbb{R}^n\setminus \cup \mathscr{F}^r}$, $\phi_{r,k}(E_k\cap \cup
	\mathscr{F}^r) \subseteq \cup\mathscr{F}_d^r $ and 
	\[
		\HM^d(\phi_{r,k}(E_k\cap \cup \mathscr{F}^r))\leq c_0\HM^d(E_k\cap \cup
		\mathscr{F}^r).
	\]
	Setting
	$A_r'=\cball_r'\cap (T+\cball(0,\tau r))$ and $H_r'=\cball_r'\cap
	(T+\cball(0,2\sqrt{n}\tau r))$, we have that $\cball_r'\setminus
	H_r'\subseteq \cup \mathscr{F}^r$ and $\phi_{r,k}\vert_{A_r'}=\id_{A_r'}$.
	By Lemma \ref{le:mp}, there exist Lipschitz mapping
	$\varphi_{r}:\mathbb{R}^n\to \mathbb{R}^n$ such that
	$\varphi_r(\cball_r)\subseteq \cball_r$,
	$\varphi_r\vert_{\mathbb{R}^n\setminus \cball_r}=\id_{\mathbb{R}^n\setminus
	\cball_r}$, $\varphi_r\vert_{H_r'}=T_{\natural}\vert_{H_r'}$ and
	\[
		\Lip(\varphi)\leq 3+2\sqrt{n}\tau/\delta\leq 3+2 \sqrt{n} \delta\leq 5.
	\]
	Hence 
	\[
		\HM^d(\varphi_r\circ \phi_{r,k}(E_k)\cap \cball_r'\setminus T)=0.
	\]

	We claim that $\varphi_r\circ \phi_{r,k}(E_k)\cap \cball_r'\cap T = T\cap
	\cball_r'$. Otherwise, we take $y\in T\cap \cball_r'\setminus \varphi_r
	\circ \phi_{r,k}(E_k)$, since $\varphi_r \circ \phi_{r,k}(E_k)$ is compact,
	there is a small ball $\oball(y,r)$ such that $\oball(y,r)\cap \varphi_r
	\circ \phi_{r,k}(E_k)=\emptyset$. Let $p_y$ the be a Lipschitz mapping such
	that $p_y\vert_{\mathbb{R}^n\setminus \oball(y,r)}=\id_{\mathbb{R}^n
	\setminus \oball(y,r)}$, $p_y(\cball_r')\subseteq \cball_r'$, and that
	$p_y\vert_{\cball_r'\setminus \oball(y,r/2)}$ is the radial projection onto
	$\partial \cball_r'$ centered at $y$. Then  
	\[
		\HM^d(p_y\circ\varphi_r\circ \phi_{r,k}(E_k\cap \cball_r))\leq
		5^d c_0\HM^d(E_k\cap \cball_r\setminus A_r'),
	\]
	and 
	\[
		\HM^d(E_k\cap \cball_r)\leq M_k(2r)\HM^d(p_y\circ\varphi_r\circ
		\phi_{r,k}(E_k\cap \cball_r))+\varepsilon_k(2r)\leq M_k(2r)5^d
		c_0\HM^d(E_k\cap \cball_r\setminus A_r')+\varepsilon_k(2r).
	\]
	Hence
	\[
		\mu(\oball_r)\leq \liminf_{k\to \infty}\HM^d(E_k\cap \oball_r)\leq 5^d c_0
		M(2r)\mu(\cball_r\setminus \intr(A_r'))+ \varepsilon(2r)
	\]
	and 
	\[
		\Theta^d(\mu,x)= \lim_{r\to 0}\frac{\mu(\oball_r)}{\omega_d r^d}\leq 5^dc_0
		M(0+)\limsup_{r\to 0}\frac{1}{r^d}(\mu(\cball_r)-\mu(\oball_r'))\leq 5^dc_0
		M(0+)\Theta^d(\mu,x)d \delta,
	\]
	this is a contradiction, so the claim holds, that is, $\varphi_r\circ 
	\phi_{r,k}(E_k)\cap \cball_r'\cap T = T\cap \cball_r'$.  Thus
	\[
		\omega_d r'^d\leq \HM^d(\phi_{r,k}(E_k)\cap \cball_r'\cap T)\leq 5^d
		c_0\HM^d(E_k\cap \cball_r\setminus A_r')+\HM^d(T_{\natural}(E_k\cap A')),
	\]
	we get that 
	\[
		\omega_d (1-\delta)^d\leq 5^dc_0\limsup_{r\to
		0}\frac{1}{r^d}(\mu(\cball_r)-\mu(\oball_r'))+\liminf_{r\to
		0}\liminf_{k\to \infty}\frac{1}{r^d}\HM^d(T_{\natural}(E_k\cap A_r')),
	\]
	and
	\[
		\liminf_{r\to 0}\liminf_{k\to \infty}\frac{\HM^d(T_{\natural}(E_k\cap
		\cball(x,r)))}{\omega_d r^d}\geq (1-\delta)^d-5^dc_0 \Theta^d(\mu,x) d
		\delta.
	\]
	Let $\delta\to 0$, we get that \eqref{eq:lpd0} holds.

\end{proof}
\begin{theorem}
	\label{thm:limitvarifold}
	Suppose that $E_k\in \QM_d(\Omega,U,M_k,\varepsilon_k)$,
	$M(r)=\limsup_{k}M_k(r)$, $\varepsilon(r)=\limsup_{k}\varepsilon_k(r)$,
	$M(0+)=1$, $\lim_{r\to 0+}r^{-d}\varepsilon(r)=0$ and $\var(E_k\cap U)\wc V$.
	For any $x\in U\cap \spt\|V\|$, if both of $\Theta^d(\|V\|,x)$ and $T=\Tan^d(\|V\|,x)$ exist, then
	\[
		\VarTan(V,x)=\{\var(T)\}.
	\]
\end{theorem}
\begin{proof}
	Without loss of generality, we assume that $x=0$, $E=U\cap \spt \|V\|$, $0\in E$,
	$T=\Tan^d(\|V\|,0)$. Write $E_k\cap U=E_{k}^{rec}\sqcup E_k^{irr}$. By Lemma 
	\ref{le:irr}, we see that $\HM^d\mr E_{k}^{irr}\wc 0$, thus $\var(E_k^{rec})
	\wc V$. 	

	For any $C\in \VarTan(V,0)$, there is a sequence $\{r_m\}$ of decreasing
	numbers with $r_m\to 0$ such that 
	\[
		C=\lim_{m\to \infty} (\scale{1/r_m})_{\sharp} V,
	\]
	thus
	\[
		C=\lim_{m\to \infty}\lim_{k\to
		\infty}(\scale{1/r_m})_{\sharp}\var(E_k^{rec})\ \text{ and
		}\ \|C\|=\lim_{m\to
		\infty}\lim_{k\to\infty}\HM^d\mr \scale{1/r_m}(E_k^{rec})
	\]
	For any $\varphi\in \cf_{c}(\mathbb{R}^n\times \grass{n}{d},\mathbb{R})$,
	we have that
	\[
		C\mr \cball(0,1)\times \grass{n}{d}(\varphi)=\lim_{m\to \infty}\lim_{k\to
		\infty}\frac{1}{r_m^d}\int_{E_k^{rec}\cap
		\oball(0,r_m)}\varphi\left(x,\Tan^d(E_{k}^{rec},x)\right)\ud \HM^d(x)
	\]

	Setting $\cball_m=\cball(0,r_m)$, since $M(0+)=1$, by Lemma \ref{le:qmtpro}
	and Lemma \ref{le:lpd}, we have that 
	\begin{equation}\label{eq:lvf9}
		\lim_{m\to \infty}\lim_{k\to \infty}\frac{1}{r_m^d}\Big(\HM^d(E_k^{rec}
		\cap \cball_m)- \HM^d\big(T_{\natural}(E_k^{rec}\cap \cball_m)\big)\Big)= 0
	\end{equation}
	and 
	\begin{equation}\label{eq:lvf10}
		\lim_{m\to \infty}\lim_{k\to \infty}\frac{1}{r_m^d}\Big(\HM^d(E_k^{rec}
		\cap \cball_m)- \HM^d(T\cap \cball_m)\Big)= 0.
	\end{equation}
	Since 
	\[
		\HM^d(T_{\natural}(E_k^{rec}\cap \cball_m))\leq \int_{x\in E_k^{rec}\cap
		\cball_m}\ap J_d(T_{\natural}\vert_{E_k^{rec}})\ud \HM^d(x),
	\]
	we get that 
	\begin{equation}\label{eq:lvf11}
		\lim_{m\to \infty}\lim_{k\to \infty}\frac{1}{r_m^d}\int_{E_k^{rec}
		\cap \cball_m}\left(1- J_d(T_{\natural}\vert_{E_k^{rec}})(x) 
		\right)\ud \HM^d(x)= 0
	\end{equation}
	and 
	\begin{equation}\label{eq:lvf12}
		\lim_{m\to \infty}\lim_{k\to \infty}\frac{1}{r_m^d} \left(\int_{E_k^{rec}
		\cap \cball_m}J_d(T_{\natural}\vert_{E_k^{rec}})(x) 
		\ud \HM^d(x) - \int_{T\cap \cball_m} 
		\ud \HM^d(x)\right) = 0.
	\end{equation}
	Setting $q_k(x)=\Tan^dE_k^{rec},x)$, by Lemma 11.4 in \cite{FK}, we have that 
	\[
		\|q_k(x)-T\|^2\leq 2 \left(1-\ap
		J_d(T_{\natural}\vert_{E_k^{rec}})(x)\right).
	\]
	Hence, setting $E_k'=E_k^{rec}$, we have that
	\[
		\begin{aligned}
			\left(\int_{E_k'\cap \cball_m}\|q_k(x)-T\| \ud \HM^d(x)\right)^2&\leq
			\int_{E_k'\cap \cball_m} \ud \HM^d(x)\int_{E_k'\cap
			\cball_m}\|q_k(x)-T\|^2 \ud \HM^d(x)\\
			&\leq 2\HM^d(E_k'\cap \cball_m)\int_{E_k'\cap \cball_m} 
			\left(1-\ap J_d(T_{\natural}\vert_{E_k'})(x)\right)\ud \HM^d(x),
		\end{aligned}
	\]
	and 
	\[
		\lim_{m\to \infty}\lim_{k\to \infty}\frac{1}{r_m^d}\int_{E_k^{rec}\cap
		\cball_m}\|q_k(x)-T\| \ud \HM^d(x)=0.
	\]
	For any $\varphi\in \cf_{c}^{\infty}(\mathbb{R}^n\times
	\grass{n}{d},\mathbb{R})$ and $T\in \grass{n}{d}$, setting $\varphi_T =
	\varphi(\cdot,T)$, we see that $\varphi_T\in
	\cf_c^{\infty}(\mathbb{R}^n,\mathbb{R})$, thus
	\[
		\|C\|\mr \cball(0,1) (\varphi_T)=\lim_{m\to \infty}\lim_{k\to
		\infty}\frac{1}{r_m^d}\int_{E_k^{rec}\cap
		\cball_m}\varphi(x,T)\HM^d(x),
	\]
	setting $C_1=C\mr \cball(0,1)\times \grass{n}{d}$ and $\mu_1 = \|C\|\mr \cball(0,1)$, we have that 
	\[
		\begin{aligned}
			|C_1(\varphi) -\mu_1(\varphi_T)|&\leq  \lim_{m\to \infty}\lim_{k\to
			\infty}\frac{1}{r_m^d}\int_{E_k^{rec}\cap
			\cball_m}\left|\varphi\left(x,\Tan^d(E_{k}^{rec},x)\right)-\varphi(x,T)
			\right|\ud\HM^d(x)\\
			&\leq \limsup_{m\to \infty}\limsup_{k\to
			\infty}\frac{1}{r_m^d}\int_{E_k^{rec}\cap
			\cball_m}\|D\varphi\|\|q_k(x)-T\|\ud\HM^d(x)=0.
		\end{aligned}
	\]
	By \eqref{eq:lvf11}, we get that 
	\[
		\lim_{m\to \infty}\lim_{k\to
		\infty}\frac{1}{r_m^d}\int_{E_k^{rec}\cap
		\cball_m}\Big(1-\ap J_d(T_{\natural}\vert_{E_k^{rec}})(x)\Big)\varphi_T(x)
		\ud\HM^d(x)=0.
	\]
	By \eqref{eq:lvf12}, we get that
	\[
		\lim_{m\to \infty}\lim_{k\to \infty}\frac{1}{r_m^d} \left|\int_{E_k^{rec}
		\cap \cball_m}\varphi_T(x)J_d(T_{\natural}\vert_{E_k^{rec}})(x) 
		\ud \HM^d(x) - \int_{T\cap \cball_m} \varphi_T(x)
		\ud \HM^d(x)\right| = 0.
	\]
	We get that 
	\[
		\begin{aligned}
			|\mu_1(\varphi_T)-\var(T_1)(\varphi)|&\leq \limsup_{m\to \infty}
			\limsup_{k\to \infty}\frac{1}{r_m^d}\left( \left| \int_{E_k^{rec}\cap
		\cball_m}\Big(1-\ap J_d(T_{\natural}\vert_{E_k^{rec}})(x)\Big)\varphi_T(x)
		\ud\HM^d(x) \right|\right.\\
			& \quad + \left.\left|\int_{E_k^{rec}
		\cap \cball_m}\varphi_T(x)J_d(T_{\natural}\vert_{E_k^{rec}})(x) 
		\ud \HM^d(x) - \int_{T\cap \cball_m} \varphi_T(x)
		\ud \HM^d(x)\right|\right) = 0.
		\end{aligned}
	\]
	Hence 
	\[
		|C_1(\varphi) -\var(T_1)(\varphi)|\leq |C_1(\varphi) -\mu_1(\varphi_T)|
		+|\mu_1(\varphi_T)-\var(T_1)(\varphi)| =0,
	\]
	and $C\mr \cball(0,1)\times \grass{n}{d}=\var(T\cap
	\cball(0,1))$. Since $C$ is a cone, we get that $C=\var(T)$
\end{proof}
\begin{theorem}
	\label{thm:limv}
	Suppose that $E_k\in \QM_d(\Omega,U,M_k,\varepsilon_k)$,
	$M(r)=\limsup_{k}M_k(r)$, $\varepsilon(r)=\limsup_{k}\varepsilon_k(r)$,
	$M(0+)<\infty$, $\lim_{r\to 0+}r^{-d}\varepsilon(r)=0$ and $\var(E_k\cap
	U)\wc V$. For any $x\in U\cap \spt \|V\|$, if $\Theta^d(\|V\|,x)=1$ and
	$T=\Tan^d(\|V\|,x)$ exists, then
	\[
		\VarTan(V,x)=\{\var(T)\}.
	\]
\end{theorem}
\begin{proof}
	Similar to the proof of Theorem	\ref{thm:limitvarifold}, we assume that
	$E=U\cap \spt \|V\|$, $x=0\in E$, $E_k\cap U=E_{k}^{rec}
	\sqcup E_k^{irr}$. 	For any $C\in \VarTan(V,0)$, there is a sequence $\{r_m\}$
	of decreasing positive numbers with $r_m\to 0$ such that 
	\[
		C=\lim_{m\to \infty} (\scale{1/r_m})_{\sharp} V.
	\]
	Since $\Theta^d(\|V\|,0)=1$, by Lemma \ref{le:lpd}, we have that
	\[
		1\leq \liminf_{m\to \infty}\liminf_{k\to \infty}\frac{\HM^d(T_{\natural}(E_k\cap
		\cball(0,r_m)))}{\omega_d r_m^d}\leq \limsup_{m\to \infty}\limsup_{k\to \infty}\frac{\HM^d(E_k\cap
		\cball(0,r_m))}{\omega_d r_m^d}\leq 1.
	\]
	Setting $\cball_m=\cball(x,r_m)$, by Lemma \ref{le:irr}, we have that
	\[
		\limsup_{m\to 0}\limsup_{k\to \infty}\frac{1}{r_m^d}\HM^d(E_k^{irr}\cap
		\cball_m)=0,
	\]
	thus
	\[
		\lim_{m\to \infty}\lim_{k\to \infty}\frac{1}{r_m^d}\Big(\HM^d(E_k^{rec}
		\cap \cball_m)- \HM^d\big(T_{\natural}(E_k^{rec}\cap \cball_m)\big)\Big)= 0
	\]
	and 
	\[
		\lim_{m\to \infty}\lim_{k\to \infty}\frac{1}{r_m^d}\Big(\HM^d(E_k^{rec}
		\cap \cball_m)- \HM^d(T\cap \cball_m)\Big)= 0.
	\]
	The rest of the proof is the same as the proof of Theorem
	\ref{thm:limitvarifold}.
\end{proof}
\begin{theorem}
	\label{thm:density}
	Suppose that $E_k\in \QM_d(\mathbb{R}^n,U,M_k,\varepsilon_k)$,
	$M(r)=\limsup_{k}M_k(r)$, $\varepsilon(r)=\limsup_{k}\varepsilon_k(r)$,
	$\HM^d \mr (E_k\cap U)\wc \mu$, $M(0+)<\infty$ and $\lim_{r\to
	0}r^{-d}\varepsilon(r)=0$. Then for $\HM^d$-a.e. $x\in
	U\cap \spt \mu$, $\Theta^d(\mu,x)$ exists and  
	\[
		1\leq \Theta^{d}(\mu,x)\leq M(0+).
	\]
\end{theorem}
\begin{proof}
	By Corollary \ref{cor:rlr}, there is a constant $c_2>0$  
	such that for $\HM^d$-a.e. $x\in E$, $\Theta^d(\mu,x)$ exists and
	\[
		1/c_2\leq\Theta^d(\mu,x)\leq c_2.
	\]
	We take $x\in U\cap \spt \mu$, such that $\Theta^d(\mu,x)\in [1/c_2,c_2]$
	and $T=\Tan^d(\mu,x)$ exists.  By Lemma \ref{le:qmtpro}, we get that 
	\[
		\limsup_{r\to 0}\limsup_{k\to \infty}\frac{1}{r^d}\HM^d(E_k^{rec}\cap
		\cball(x,r))\leq M(0+)\omega_d.
	\]
	Since 
	\[
		\limsup_{k\to \infty}\HM^d(E_k^{irr}\cap \cball(x,r))\leq
		\varepsilon(2r) 
	\]
	and $\lim_{r\to 0}r^{-d}\varepsilon(2r)=0$,
	we get that
	\[
		\Theta^d(\mu,x)=\lim_{r\to 0}\frac{\mu(\oball(x,r))}{\omega_d
		r^d}\leq \limsup_{r\to 0}\limsup_{k\to \infty}\frac{1}{\omega_d r^d}
		\HM^d(E_k\cap \oball(x,r))\leq M(0+).
	\]

	By Lemma \ref{le:lpd}, we get that 
	\[
		\begin{aligned}
			\lim_{r\to 0}\frac{\mu(\cball(x,r))}{\omega_d r^d}&\geq \liminf_{r\to
			0}\limsup_{k\to \infty}\frac{1}{\omega_d r^d}\HM^d(\HM^d(E_k\cap\cball(x,r)))\\
			&\geq \liminf_{r\to 0}\liminf_{k\to \infty}\frac{1}{\omega_d
			r^d}\HM^d(\HM^d(T_{\natural}(E_k\cap\cball(x,r))))\geq 1.
		\end{aligned}
	\]

\end{proof}
\begin{corollary}[Theorem 3.4 in \cite{David:2003} and Lemma 3.12 in
	\cite{David:2009}]\label{co:lusc}
	Suppose that $E_k\in \QM_d(\mathbb{R}^n,U,M_k,\varepsilon_k)$,
	$M(r)=\limsup_{k}M_k(r)$, $\varepsilon(r)=\limsup_{k}\varepsilon_k(r)$ 
	and $\HM^d \mr (E_k\cap U)\wc \mu$, $E=\spt \mu$. If $M(0+)<\infty$ and $\lim_{r\to
	0}r^{-d}\varepsilon(r)=0$, then we have that  for any open set $O\subseteq
	U$,   
	\[
		\HM^d(E\cap O)\leq \liminf_{k\to \infty}\HM^d(E_k\cap O),
	\]
	and for any compact set $H\subseteq U$, 
	\[
		\limsup_{k\to \infty}\HM^d(E_k\cap H)\leq M(0+)\HM^d(E\cap H).
	\]
\end{corollary}

\begin{corollary}
	Suppose that $E_k\in \QM_d(\mathbb{R}^n,U,M_k,\varepsilon_k)$,
	$M(r)=\limsup_{k}M_k(r)$, $\varepsilon(r)=\limsup_{k}\varepsilon_k(r)$ 
	$\var(E_k\cap U)\wc V$, $E=\spt(\|V\|)$ and $\lim_{r\to
	0}r^{-d}\varepsilon(r)=0$.
	\begin{itemize}
		\item If $M(0+)=1$, then $V\mr U\times\grass{n}{d}=\var(E\cap U)$.
		\item If $M(0+)<\infty$ and there is an open set $O\subseteq U$ such that 
			\[
				\HM^d(E\cap O)=\lim_{k\to \infty}\HM^d(E_k\cap O),
			\]
			then $V\mr O\times\grass{n}{d}=\var(E\cap O)$.
	\end{itemize}
\end{corollary}
\begin{proof}
	If $M(0+)=1$, then the conclusion directly follows from Theorem
	\ref{thm:limitvarifold} and Proposition \ref{prop:limitradon}.
	If $M(0+)<\infty$ and $\HM^d(E\cap O)=\lim_{k\to \infty}\HM^d(E_k\cap O)$,
	by Theorem \ref{thm:density}, we see that $\Theta^d(\mu,x)=1$ for
	$\HM^d$-a.e. $x\in E\cap O$, then by Theorem \ref{thm:limv} and Proposition
	\ref{prop:limitradon}, we get the conclusion.
\end{proof}

\section{Qusaiminimal sets on $C^2$ submanifold}
For any nonempty set $A\subseteq \mathbb{R}^n$, we denote by
$\delta=\delta_A:\mathbb{R}^n\to \mathbb{R}$ the distant function defined by
$\delta(x)=\dist(x,A)$. We denote by $\Unp(A)$ the set of points $x\in \mathbb{R}^{n}$
such that there is a unique point $\xi_A(x)\in A$ satisfying that
$\delta(x)=|x-\xi_A(x)|$. Indeed, the mapping $\xi_A:\Unp(A)\to A$ is called
metric projection. For any $z\in A$, we define $\reach(A,z)$ to be the
supremum of numbers $r>0$ for which $\oball(x,r)\subseteq \Unp(A)$, and define 
the reach of $A$ by $\reach(A)=\inf\{\reach(A,z):z\in A\}$. 
By Remark 4.2 in \cite{Federer:1959}, we see that $\reach(A,\cdot)$ is
continuous on $A$. For any set
$W\subseteq \Unp(A)$, setting $\rho=\sup\{\dist(x,A):x\in W\}$ and
$R=\inf\{\reach(A,\xi(x)):x\in W\}$, by Theorem 4.8 (8) in \cite{Federer:1959},
we have that 
\begin{equation}\label{eq:mplip}
	\Lip(\xi\vert_{W})\leq \frac{R}{R-\rho}.
\end{equation}

\begin{lemma}\label{le:retr}
	Suppose $1\leq d\leq m\leq n$.
	Let $\Omega\subseteq \mathbb{R}^n$ be an $m$-dimensional closed submanifold 
	of class $C^2$. Let $U\subseteq \mathbb{R}^n$ be a bounded open set with
	$U\cap \Omega\neq \emptyset$ and $\overline{U}\cap \partial \Omega=\emptyset$. 
	If $E\in \QM_d(\Omega,U,M,\varepsilon)$, then there exist $r_0=r_0(\Omega,U)>0$
	and two nondecreasing functions $M'$ and $\varepsilon'$ such that $E\in
	\QM_d(\mathbb{R}^n,U,M', \varepsilon')$, $M'(r)\leq (1+r/r_0)^dM(r)$ and
	$\varepsilon'(r)\leq \varepsilon(r)$ for $0<r\leq r_0$.
\end{lemma}
\begin{proof}
	We take $\xi=\xi_{\Omega}$ and
	$r_0=\frac{1}{4}\inf\{\reach(\Omega,\xi(x)):x\in U\cap \Unp(\Omega)\}$.
	Since $\overline{U}\cap \partial \Omega=\emptyset$, and
	$\reach(\Omega,\cdot)$ is continuous on $\Omega$, we get that $r_0>0$. 
	For any $\varphi\in \mathcal{D}(\mathbb{R}^n,U)$ with
	$\diam(W_{\varphi}\cap \varphi(W_{\varphi}))\leq r_0$, we assume
	$W_{\varphi}\cap E\neq \emptyset$, otherwise we have nothing to do. Then
	$W=W_{\varphi}\cup \varphi(W_{\varphi})\subseteq \Unp(\Omega)$ and
	$\sup\{\dist(x,\Omega):x\in W\}\leq \diam(W)\leq r_0$, thus we get that
	\[
		\Lip(\xi\vert_{W})\leq \frac{2r_0}{2r_0-\diam(W)}\leq
		1+\frac{\diam(W)}{r_0}.
	\]
	Since $\xi\circ \varphi\in \mathcal{D}(\Omega,U)$, we get that 
	\[
		\begin{aligned}
			\HM^d(E\cap W_{\varphi})&\leq M(\diam W)\HM^d(\varphi(E\cap
			W_{\varphi}))+\varepsilon(\diam W)\\
			&\leq M(\diam W) \Lip(\xi_{W})^d\HM^d(\varphi(E\cap W_{\varphi}))+
			\varepsilon(\diam W)\\
			&\leq  M(\diam W) (1+r_0^{-1}\diam W)^d\HM^d(\varphi(E\cap W_{\varphi}))+
			\varepsilon(\diam W).
		\end{aligned}
	\]
	The conclusion holds with $M'(r)=(1+r/r_0)^dM(r)$ and
	$\varepsilon'(r)=\varepsilon(r)$ for $0<r\leq r_0$.
\end{proof}
\begin{theorem} \label{thm:vcp}
	Suppose $1\leq d\leq m\leq n$.
	Let $\Omega\subseteq \mathbb{R}^n$ be a closed $m$-dimensional submanifold 
	of class $C^2$. Let $U\subseteq
	\mathbb{R}^n$ be an open set such that $U\cap \Omega\neq \emptyset$ and
	$U\cap \partial \Omega=\emptyset$. Suppose that
	$\{E_k\}\subseteq \QM_d(\Omega,U,M_k,\varepsilon_k)$, $M(r)=\limsup_{k\to
	\infty}M_k(r)$, $\varepsilon(r)=\limsup_{k\to \infty}\varepsilon_k(r)$,
	$\HM^d\mr (E_k\cap U)\wc \mu$, $M(0+)<\infty$ and $\lim_{r\to 0+}
	r^{-d}\varepsilon(r)=0$.  Then we have that
	\begin{itemize}
		\item $U\cap \spt\mu$ is $d$-rectifiable, for $\HM^d$-a.e. $x\in U\cap
			\spt \mu$, $\Theta^d(\mu,x)$ exists and 
			\[
				1\leq \Theta^d(\mu,x)\leq M(0+);
			\]
		\item if $O\subseteq U$ is an open set and $\Theta^d(\mu,x)=1$ for
			$\mu$-a.e. $x\in O$, then $\var(O\cap E_k)$ converges to a varifold
			$V$
			with $V\mr O\times \grass{n}{d}=\var(O\cap \spt \mu)$. 
	\end{itemize}
	In particular, if $M(0+)=1$ and $\var(U\cap E_k)\wc V$, then $V\mr
	U\times\grass{n}{d}=\var(U\cap \spt \|V\|)$.
\end{theorem}
\begin{proof}
	For any closed ball $\cball(x,r)\subseteq U$, if $\oball(x,r)\cap \Omega\neq
	\emptyset$, by Lemma \ref{le:retr}, we see that 
	\[
		E_k\in \QM_d(\mathbb{R}^n,\oball,M_k',\varepsilon_k'),
	\]
	where $\oball=\oball(x,r)$, $M_k'(r)=(1+r/r_0)^dM_k(r)$ and
	$\varepsilon_k'(r)= \varepsilon_k(r)$ for $0<r\leq r_0$,
	$r_0=r_0(\Omega,\oball)>0$. Put $M(r)'=\limsup_{k\to \infty}M_k'(r)$ and
	$\varepsilon(r)'=\limsup_{k\to \infty}\varepsilon_k'(r)$. Then
	$M'(0+)=M(0+)<\infty$ and $\lim_{r\to 0+}r^{d}\varepsilon'(r)=\lim_{r\to
	0+}r^{d}\varepsilon(r)=0$. By Theorem \ref{thm:density}, we get that
	$\oball\cap \spt \mu$ is $d$-rectifiable, for $\HM^d$-a.e. $x\in \oball$,
	$\Theta^d(\mu,x)$ exists and $1\leq \Theta^d(\mu,x)\leq M(0+)$.

	For any closed ball $\cball(x,r)\subseteq O$, by Theorem \ref{thm:limv}, we
	get that $\var(O\cap E_k)$ converges to a varifold $V$, which satisfies
	\[
		V\mr O\times \grass{n}{d}=\var(O\cap \spt \|V\|).
	\]
\end{proof}

\section{Plateau's Problem}
Let $\Omega \subseteq \mathbb{R}^n$ be a nonempty closed subset, $B_0\subseteq
\Omega$ be a nonempty compact subset.  Let $G$ be an abelian group. Suppose 
that $L\subseteq \cech_{d-1}(B_0;G)$ is a subgroup. A compact set $E\subseteq
\Omega$ is called spanning $L$ (or whose  algebraic boundary contains $L$) if
$E\supseteq B_0$ and $\cech_{d-1}(i_{B_0,E})(L)=0$, where $i_{B_0,E}:B_0\to E$
is the inclusion mapping and $\cech_{d-1}(i_{B_0,E})$ is the homomorphism
induced by $i_{B_0,E}$. We denote by $\check{C}=\check{C}(\Omega, B_0,G,L)$ the
collection of compact subsets in $\Omega$ which are spanning $L$. Elements in
$\check{C}$ are called competitors for the Plateau's problem with \v{C}ech
homological conditions. A sequence $\{E_k\}\subseteq \check{C}$ is called a
minimizing sequence if $\HM^d(E_k\setminus B_0)\to \inf\{\HM^d(E\setminus
B_0):E\in \check{C}\}$. If there is a compact set $E\in \check{C}$ such that 
$\HM^d (E\setminus B_0)= \inf\{\HM^d(E\setminus B_0):E\in \check{C}\}$, then
we call $E$ a minimizer. It is possible $\check{C}=\emptyset$, if we do not
presuppose any condition on $\Omega$ and $B_0$. It is quite easy to see that, 
for any $E\in \check{C}$ and $\varphi\in \mathcal{D}(\Omega,U)$, $\varphi(E)
\in \check{C}$, where $U$ is any open set with $U\cap B_0=\emptyset$.

\begin{lemma}\label{le:CI}
	Let $A_1$ and $A_2$ be two compact subsets in $\mathbb{R}^n$. If
	$\cech_{d-1}(A_1\cap A_2;G)=0$, then the homomorphism
	$\cech_{d-1}(i_{A_1,A_1\cup A_2})$ is injective. 
\end{lemma}
\begin{proof}
	By \cite[Theorem 15.3 in p.39]{ES:1952} and the fact that every compact
	triad is a proper triad \cite[p.257]{ES:1952}, we get that the
	Mayer-Vietoris sequence of triad $(A_1\cup A_2; A_1,A_2)$ for {\v C}ech
	homology is exact. That is, it holds the following exact sequence:
	\[
		\cdots \to \cech_{d-1}(A_1\cap A_2;G)\ora{\psi}
		\cech_{d-1}(A_1;G)\oplus\cech_{d-1}(A_2;G)\ora{\phi}
		\cech_{d-1}(A_1\cup A_2;G)\to
		\cdots\to 0,
	\]
	where homomorphism $\psi$ and $\phi$ are defined by
	\[
		\psi(u)=\cech_{d-1}(i_{A_1\cap
		A_2,A_1})(u)-\cech_{d-1}(i_{A_1\cap A_2,A_2})(u)
	\] 
	and
	\[
		\phi(v_1,v_2)=\cech_{d-1}(i_{A_1,A_1\cap
		A_2})(v_1)+\cech_{d-1}(i_{A_2,A_1\cap A_2})(v_2).
	\]

	Since $\cech_{d-1}(A_1\cap A_2;G)=0$, we get that $\phi$ is injective.
	Let
	\[
		j:\cech_{d-1}(A_1;G)\to\cech_{d-1}(A_1;G)\oplus\cech_{d-1}(A_2;G)
	\]
	be the homomorphism defined by $j(v)=(v,0)$. Then 
	$ \cech_{d-1}(i_{A_1,A_1\cup A_2})=\phi\circ j $ is injective.

\end{proof}
\begin{lemma}\label{le:cswl}
	Let $B_0\subseteq \mathbb{R}^n$ be a nonempty compact subset.
	Let $G$ be an abelian group. Suppose that $L\subseteq \cech_{d-1}(B_0;G)$ is
	a subgroup. Let $\{E_k\}\subseteq \check{C}(\mathbb{R}^n,B_0,G,L)$ be a 
	sequence of compact subsets, which is uniformly bounded. Suppose that
	$\HM^d\mr (E_k\setminus B_0)\wc \mu$. Then $B_0\cup \spt \mu$ is spanning $L$.
\end{lemma}
\begin{proof}
	We assume $B_0\cup \cup_k E_k\subseteq \cball(0,R)$, $R>0$, put $E=B_0\cup \spt 
	\mu$ and $U=\oball(0,R+1)\setminus E$. Then $E$ is compact and contained in
	$\cball(0,R)$, $U$ is open. Let $\mathscr{F}$ be the family of cubes
	which raise from a Whitney decomposition of $U$ with
	the following conditions hold:
	\begin{itemize}
		\item $\mathscr{F}$ consists of interior disjoint dyadic cubes  and 
			$\cup_{Q\in \mathscr{F}} Q=U$;
		\item $\sqrt{n}\ell(Q)\leq \dist(Q,\mathbb{R}^n\setminus U)\leq 
			4\sqrt{n}\ell(Q)$;
		\item if $Q_1\cap Q_2\neq \emptyset$, then $1/4\leq
			\ell(Q_1)/\ell(Q_2)\leq 4$.
	\end{itemize}

	We put $F_k=E+\oball(0,4^{-k+1})$ and let $\mathcal{Q}^{k}$ be the collection of
	cubes $Q$ in $\mathscr{F}$ whose sidelength is no less than $4^{-k}$.  
	Setting $G_k=\bigcup\{Q:Q\in \mathcal{Q}^k\}$, we have that $G_k
	\subseteq \intr(G_{k+1})$ and $\mathbb{R}^n\setminus\intr(G_{k})\subseteq F_k$.

	Since $\mu(U)=0$, there exists a increasing sequence $\{m_k\}$ of positive
	integers such that 
	\[
		\HM^d(E_m\cap G_k)<c_0^{-1} 4^{-kd},\ \forall m\geq m_k .
	\]
	We take $E_k'=E_{m_k}$.	By Theorem \ref{thm:ffp}, there exist Lipschitz mappings
	$\phi_k:\mathbb{R}^n\to \mathbb{R}^n$ such that $\phi_k(Q)\subseteq Q$ for
	$Q\in \mathcal{Q}^k$, $\phi_k\vert_{\mathbb{R}^n\setminus G_k}=
	\id_{\mathbb{R}^n\setminus G_k}$ and $\intr(G_k)\cap \phi_k(E_k')\subseteq 
	\cup \mathcal{Q}_{d-1}^k$. Since $E_k'$ is spanning $L$, we get that $\phi_k(E_k')$
	is spanning $L$. Setting $E_k''=\phi_k(E_k')\setminus \intr(G_k)$ and
	$H_k=\overline{\intr(G_k)\cap \phi_k(E_k')}$, we see that $H_k\cap
	E_k''\subseteq \cup \mathcal{Q}_{d-2}^k$ and $E_k''\cup H_k=\phi_k(E_k')$,
	thus $\HM^{d-1}(H_k\cap E_k'')=0$. Applying Theorem \Rnum{7} 3 in 
	\cite{HW:1941}, we get that 
	\[
		\cech_{d-1}(H_k\cap E_k'';G)=0.
	\]
	By lemma \ref{le:CI}, we get that the homomorphism
	$\cech_{d-1}\big(i_{E_k'',\phi_k(E_k')}\big)$ is injective.
	Since 
	\[
		i_{B_0,\phi_k(E_k')}=i_{E_k'',\phi_k(E_k')}\circ i_{B_0,E_k''},
	\]
	we have that 
	\[
		\cech_{d-1}\big(i_{B_0,\phi_k(E_k')}\big)=\cech_{d-1}
		\big(i_{E_k'',\phi_k(E_k')}\big)\circ \cech_{d-1}\big(i_{B_0,E_k''}\big).
	\]
	Since $L\subseteq \ker\left( \cech_{d-1}\big(i_{B_0,\phi_k(E_k')}\big) \right)$ and
	that $\cech_{d-1}\big(i_{E_k'',\phi_k(E_k')}\big)$ is injective, we get that 
	\[
		L\subseteq \ker\left( \cech_{d-1}\big(i_{B_0,E_k''}\big) \right).
	\]
	Since $\mathbb{R}^n\setminus \intr(G_k)\subseteq F_k$, we get that
	$E''_k\subseteq F_k$, thus  
	\[
		L\subseteq \ker\left( \cech_{d-1}\big(i_{B_0, F_k}\big) \right).
	\]
	By Lemma 12.2 in \cite{FK}, we get that 
	\[
		L\subseteq \ker\left( \cech_{d-1}\big(i_{B_0, E}\big) \right).
	\]

\end{proof}
\begin{proof}[Proof of Theorem \ref{thm:RPM}]
	Let $\{E_k\}$ be a uniformly bounded minimizing sequence. That is, $E_k\in \check{C}$ and 
	\[
		\HM^d(E_k\setminus B_0)\to \inf\{\HM^d(E\setminus B_0):E\in \check{C}\}.
	\]
	Without loss of generality, we assume that $0<\inf\{\HM^d(E\setminus B_0):E\in
	\check{C}\}<\infty$.
	Put $U=\Omega\setminus B$, $\zeta_k=\inf\big\{\HM^d(E_k)-\HM^d(\varphi(E_k)):
	\varphi\in \mathcal{D}(\Omega,U))\big\}$, $M_k\equiv 1$,
	$\varepsilon_k\equiv \zeta_k$, $M(r)=\limsup_{k\to \infty}M_k(r)$ and
	$\varepsilon(r)=\limsup_{k\to \infty}\varepsilon_k(r)$. Then $M\equiv 1$,
	$\varepsilon\equiv 0$ and 
	\[
		E_k\in \QM_d(\Omega,U,M_k,\varepsilon_k).
	\]
	Since $\{E_k\}$ is uniformly bounded and $\inf\{\HM^d(E\setminus B_0):E\in
	\check{C}\}<\infty$, we can find a subsequence $\{E_{k_m}\}$ such that
	$\HM^d\mr (U\cap E_{k_m})$ converges to a Radon measure $\mu$. By Theorem
	\ref{thm:vcp}, we get that $\var(U\cap E_{k_m})\wc V$ and $V\mr U\times
	\grass{n}{d}=\var(U\cap \spt \mu)$. By Lemma \ref{le:cswl}, we get that
	$E_{\infty}=B_0\cup \spt \mu\in \check{C}$. Thus 
	\[
		\HM^d(E_{\infty}\setminus B_0)= \HM^d(U\cap \spt\mu)\leq \liminf_{m\to
		\infty}\HM^d(U\cap E_{k_m})=\inf\{\HM^d(E\setminus B_0):E\in \check{C}\}.
	\]
	Therefore, $E_{\infty}$ is a minimizer.
\end{proof}

\begin{lemma}\label{le:hc}
	Let $U\subseteq \mathbb{R}^n$ be an open set.  Let $\{ E_k\}$ be a sequence
	of closed subsets such that $U\cap E_k$ is uniformly bounded. Suppose that
	$\limsup_{k\to \infty}\HM^d(U\cap E_k)<\infty$ and $\HM^d\mr (U\cap E_k)\wc
	\mu$. Then, for any $\tau>0$, there exist a subsequence $\{E_{k_i}\}$ and a
	sequence of Lipschitz mappings $\{\phi_i\}\subseteq  \mathcal{D} 
	(\mathbb{R}^n,U)$ such that $\|\phi_i-\id\|\leq \tau$, $\phi_i(U\cap E_{k_i})$
	converges to a compact set $E$ in $U$ in local Hausdorff distance,
	\begin{equation}\label{eq:hc0}
		\lim_{i\to\infty}\HM^d(\phi_i(E_{k_i})\, \triangle\, E_{k_i})=0,\
		U\cap \spt\mu\subseteq E\text{ and }\HM^d (U\cap E\setminus \spt\mu)=0.
	\end{equation}
\end{lemma}

\begin{proof}
	We assume that $0<\tau<1/10$ and $E_k\cap U\subseteq \cball(0,R)$ for any 
	$i\geq 1$and some $R>0$, put $U=\oball(0,R+1)\setminus \spt \mu$
	and $E=U\cap \spt \mu$. Then $U$ is open. Let
	$\widetilde{\mathscr{F}}$ be the family of cubes
	which raise from a Whitney decomposition of $U$ with
	the following conditions hold:
	\begin{itemize}
		\item $\widetilde{\mathscr{F}}$ consists of interior disjoint dyadic cubes  and 
			$\cup_{Q\in \widetilde{\mathscr{F}}} Q=U$;
		\item $\sqrt{n}\ell(Q)\leq \dist(Q,\mathbb{R}^n\setminus U)\leq 
			4\sqrt{n}\ell(Q)$;
		\item if $Q_1\cap Q_2\neq \emptyset$, then $1/4\leq
			\ell(Q_1)/\ell(Q_2)\leq 4$.
	\end{itemize}
	We construct $\mathscr{F}$ from $\widetilde{\mathscr{F}}$ as follows: for 
	any $C\in \widetilde{\mathscr{F}}$, if its sidelength is less than $\tau$, 
	then we put $C\in \mathscr{F}$; otherwise, we decompose it into diaydic
	cubes of sidelength $2^{[\ln \tau /\ln 2]}$, then put each smaller cubes
	into $\mathscr{F}$. 
	We put $F_k=E+\oball(0,4^{-k+1})$ and let $\mathcal{Q}^{k}$ be the collection of
	cubes $Q$ in $\mathscr{F}$ whose sidelength is no less than $4^{-k}$.  
	By putting $G_k=\bigcup\{Q:Q\in \mathcal{Q}^k\}$, we have that $G_k
	\subseteq \intr(G_{k+1})$ and $\mathbb{R}^n\setminus\intr(G_{k})\subseteq F_k$.

	Since $\mu(U)=0$, there exists a increasing sequence $\{k_i\}$ of positive
	integers such that 
	\[
		\HM^d(E_m\cap G_i)<c_0^{-1} 4^{-id},\ \forall m\geq k_i .
	\]
	By Theorem \ref{thm:ffp}, there exist Lipschitz mappings
	$\phi_i:\mathbb{R}^n\to \mathbb{R}^n$ such that $\phi_i(Q)\subseteq Q$ for
	$Q\in \mathcal{Q}^i$, $\phi_i\vert_{\mathbb{R}^n\setminus G_i}=
	\id_{\mathbb{R}^n\setminus G_i}$ and $\intr(G_i)\cap \phi_i(E_{k_i})\subseteq 
	\cup \mathcal{Q}_{d-1}^i$. Since $\cup G_i= U$, $\mathbb{R}^n\setminus
	\intr(G_i)\subseteq E+\oball(0,4^{-i+1})$ and $ \phi_i(E_{k_i})\cap 
	\intr(G_i)$ is contained in the union of $(d-1)$-faces of cubes in
	$\mathscr{F}$, we get that $\phi_i(U\cap E_{k_i})$ converges to $E$ in $U$ in 
	local Hausdorff distance, and \eqref{eq:hc0} hold.
\end{proof}
\begin{theorem}
	Suppose $1\leq d\leq m\leq n$.
	Let $\Omega\subseteq \mathbb{R}^n$ be an $m$-dimensional closed submanifold 
	of class $C^2$. Let $U\subseteq \mathbb{R}^n$ be an open set with $U\cap 
	\Omega\neq \emptyset$ and $U\cap \partial \Omega=\emptyset$. Let
	$\mathcal{G}(\Omega,U)$ be a collection of compact subsets in $\Omega$ such that
	\begin{itemize}
		\item[(a)] for any $E\in \mathcal{G}(\Omega,U)$ and $\varphi\in
			\mathcal{D}(\Omega,U)$ we have that $\varphi(E)\in
			\mathcal{G}(\Omega,U)$;
		\item[(b)] there is a sequence $\{E_i\}\subseteq \mathcal{G}(\Omega,U)$ 
			such that $\{E_i\cap U\}$ is uniformly bounded and 
			\[
				\HM^d(E_i\cap U)\to
				\inf\{\HM^d(S\cap U):S\in \mathcal{G}(\Omega,U)\}.
			\]
	\end{itemize}
	Suppose that one of the following conditions hold:
	\begin{itemize}
		\item[(c)] For any sequence $\{E_i\}\subseteq \mathcal{G}(\Omega,U)$, if
			$\HM^d \mr (E_i\cap U)\wc \mu$, then 
			$\spt \mu \in \mathcal{G}(\Omega,U)$.
		\item[(d)] For any sequence $\{E_i\}\subseteq \mathcal{G}(\Omega,U)$, if 
			$E_i\cap U$ converges to $E$ in $U$ in local Hausdorff distance, then 
			$E\in \mathcal{G}(\Omega,U)$.
	\end{itemize}
	Then there exists $E\in \mathcal{G}(\Omega,U)$ such that 
	\[
		\HM^d(E\cap U)=\inf\{\HM^d(S\cap U):S\in \mathcal{G}(\Omega,U)\}.
	\]
\end{theorem}
\begin{proof}
	We assume that $\inf\{\HM^d(S\cap U):S\in \mathcal{G}(\Omega,U)\}<\infty$,
	otherwise it holds trivially. Let $\{E_i\}\subseteq\mathcal{G}(\Omega,U)$ be 
	a uniformly bounded sequence such that 
	\[
		\HM^d(E_i\cap U)\to
		\inf\{\HM^d(S\cap U):S\in \mathcal{G}(\Omega,U)\}.
	\]
	We assume that $E_i\cap U\subseteq \cball(0,R)$ for all $i\geq 1$ and some $R>0$.

	If $(c)$ holds, we take $E=\spt\mu$ and put $\zeta_k=
	\inf\big\{\HM^d(E_k)-\HM^d(\varphi(E_k)):
	\varphi\in \mathcal{D}(\Omega,U))\big\}$, $M_k\equiv 1$,
	$\varepsilon_k\equiv \zeta_k$, $M(r)=\limsup_{k\to \infty}M_k(r)$ and
	$\varepsilon(r)=\limsup_{k\to \infty}\varepsilon_k(r)$, then we have that
	$E\in \mathcal{G}(\Omega,U)$, $M\equiv 1$, $\varepsilon\equiv 0$ and 
	\[
		E_k\in \QM_d(\Omega,U,M_k,\varepsilon_k).
	\]
	By Theorem \ref{thm:vcp}, we get that $\Theta^d(\mu,x)=1$ for $\HM^d$-a.e.
	$x\in E\cap U$, thus 
	\[
		\HM^d(E\cap U)\leq \liminf_{k\to \infty}\HM^d(E_k\cap U) =
		\inf\{\HM^d(S\cap U):S\in \mathcal{G}(\Omega,U)\}.
	\]
	Hence 
	\[
		\HM^d(E\cap U) = \inf\{\HM^d(S\cap U):S\in \mathcal{G}(\Omega,U)\},
	\]
	and $E$ is our desired minimizer.

	If $(d)$ holds, then we assume, up to a subsequence, that $\HM^d\mr(E_k\cap
	U)\wc \mu$. By Lemma \ref{le:hc}, for any $\tau>0$, we can find a 
	subsequence $\{E_{k_i}\}$ of $ \{E_k\}$ and a sequence of Lipschitz mappings
	$\{\phi_i\}\subseteq \mathcal{D}(\mathbb{R}^n, U)$ such that 
	$\|\phi_i-\id\|\leq \tau$, $\phi_i(U\cap E_{k_i})$ converges to a set $E$ in
	$U$ in local Hausdorff distance. We will show that $E$ is our desired
	minimizer. Indeed, we see that $E\in \mathcal{G}(\Omega,U)$. Since
	$\HM^d(U\cap E\setminus \spt\mu)=0$, we get that 
	\[
		\HM^d(E\cap U)\leq \HM^d(U\cap \spt\mu).
	\]
	Similar to the above case, by Theorem \ref{thm:vcp}, we have that
	\[
		\HM^d(E\cap U)\leq \HM^d(U\cap \spt\mu)\leq \inf\{\HM^d(S\cap U):S\in \mathcal{G}(\Omega,U)\},
	\]
	and finally we get that $E$ is a minimizer.
\end{proof}
\begin{remark}
	Let $\Omega\subseteq \mathbb{R}^n$ be an $m$-dimensional closed submanifold
	of class $C^2$, $B_0$ be a compact subset, $\check{C}(\Omega, B_0,G,L)$ be the collection of subsets in $\Omega$ which
	are spanning $L$. We take $U=\mathbb{R}^n\setminus B_0$ and $\mathcal{G}
	(\Omega,U)=\check{C}(\Omega, B_0,G,L)$, then $\mathcal{G} (\Omega,U)$
	satisfies both conditions $(c)$ and $(d)$, as well as condition $(a)$, in above theorem. 
\end{remark}
\section{Qusaiminimality of limit sets}
\begin{lemma}\label{le:ld}
	Let $E\subseteq \mathbb{R}^n$ be a $d$-rectifiable Borel set with
	$0<\HM^d(E)<\infty$, $\varphi:\mathbb{R}^n\to \mathbb{R}^n$ be a Lipschitz
	mapping. Then for $\HM^d$-a.e. $x\in E$, the approximate tangent plane
	$T_x=\Tan^dE,x)$ exists, and   
	\begin{equation}\label{eq:ld}
		\lim_{r\to 0}\frac{\HM^d(\varphi(E\cap
		\cball(x,r)))}{\omega_d r^d}=\lim_{r\to 0}\frac{\HM^d(\varphi(T_x\cap
		\cball(x,r)))}{\omega_d r^d}=\ap J_d(\varphi\vert_E)(x).
	\end{equation}
\end{lemma}
\begin{proof}
	If $\varphi\vert_{E}$ is injective, by Theorem 3.2.22 in
	\cite{Federer:1969}, we have that
	\[
		\HM^d(\varphi(E\cap \cball(x,r)))=\int_{E\cap \cball(x,r)}\ap
		J_d(\varphi\vert_{E})(y)\ud \HM^d(y)=\int_{\cball(x,r)}\ap
		J_d(\varphi\vert_{E})\ud\HM^d\mr E.
	\]
	Since $E$ is $d$-rectifiable, we see that for $\HM^d$-a.e. $x\in E$
	\[
		\lim_{r\to 0}\frac{\HM^d(E\cap \cball(x,r))}{\omega_d r^d}=1.
	\]
	By Lebesgue density theorem, we get that for $\HM^d$-a.e. $x\in E$,
	\[
		\lim_{r\to 0}\frac{\HM^d(\varphi(E\cap
		\cball(x,r)))}{\omega_d r^d}=\lim_{r\to 0}
		\frac{\int_{\cball(x,r)}\ap J_d(\varphi\vert_{E})\ud\HM^d\mr E}
		{\HM^d\mr E (\cball(x,r))}=\ap J_d(\varphi\vert_{E})(x).
	\]
	Since $E$ is $d$-rectifiable and $\varphi$ is Lipschitz, we see that 
	$\HM^d$-a.e. $x\in E$, $T_x=\Tan^dE,x)$ and $\ap D(\varphi\vert_E)(x)$ exist,
	then for such $x$, $D(\varphi\vert_{T_x})(x)$ exist and
	$D(\varphi\vert_{T_x})(x)=\ap D(\varphi\vert_E)(x)$, thus 
	\[
		\lim_{r\to 0}\frac{\HM^d(\varphi(T_x\cap \cball(x,r)))}{\omega_d r^d}
		=\|\wedge_d D(\varphi\vert_{T_x})(x)\|=\ap J_d(\varphi\vert_{E})(x).
	\]
	We get that \eqref{eq:ld} hold when $\varphi$ is injective.

	In general case, we let $E'\subseteq E$ be a Borel set such that 
	$\HM^d(E\setminus E')=0$, and  for any $x\in E'$, $\ap D(\varphi\vert_E)(x)$
	exists. By Corollary 3.2.4 in \cite{Federer:1969}, we have the decomposition 
	\[
		E'=E_0\cup \bigcup_{i\geq 1} E_i,
	\]
	where $E_i$, $i\geq 0$, are mutual disjoint Borel sets, 
	$J_d (\varphi\vert_E)(x)=0$ for $x\in E_0$, $\varphi\vert_{E_i}$ is injective
	for $i\geq 1$. Put $\widehat{E}_i=E'\setminus E_i$ for $i\geq 1$. We see
	that $\HM^d(\varphi(E_0))=0$, for $\HM^d$-a.e. $x\in E_i$,
	$\Theta^d(E_i,x)=1$, $\Theta^d(\widehat{E}_i,x)=0$,
	$\Tan^dE_i,x)=\Tan^dE,x)$ exists, and $\ap J_d(\varphi\vert_{E_i})(x)=\ap
	J_d(\varphi\vert_E)(x)$. Since $\varphi\vert_{E_i}$ is injective and
	$\varphi$ is Lipschitz, we get that for $\HM^d$-a.e. $x\in E_i$,
	\[
		\lim_{r\to 0}\frac{\HM^d(\varphi(E\cap
		\cball(x,r)))}{\omega_d r^d}=\lim_{r\to 0}\frac{\HM^d(\varphi(T_x\cap
		\cball(x,r)))}{\omega_d r^d}=\ap J_d(\varphi\vert_{E_i})(x)=\ap J_d
		(\varphi\vert_E)(x).
	\]

\end{proof}

\begin{proof}[Proof of Theorem \ref{thm:QM}]
	Let $\xi=\xi_{\Omega}$ be the metric projection from $\Unp(\Omega)$ onto
	$\Omega$. Let $\varphi\in \mathcal{D}(\Omega,U)$ be any deformation. We take 
	$0<\tau<1/2$, and put $W=W_{\varphi}=\{x\in U:\varphi(x)\neq x\}$. 
	For any $\delta>0$, we put $W^{\delta}=\{x\in
	W:\dist(x,\mathbb{R}^n\setminus W)>\delta\}$. Since $\overline{W}\subseteq U$
	is compact and $\reach(\Omega,\cdot)$ is continuous, we have that
	$r_0=\inf\{\reach(\Omega,x):x\in W\}>0$. For any $0<\delta<r_0/2$, we put 
	$O_{\delta}=W^{3 \delta}+\cball(0,\delta)$, by \eqref{eq:mplip}, then we 
	have that $\Lip(\xi\vert_{O_{\delta}})\leq 2$. Indeed, if $x\in W$ and
	$0<r<r_0/2$, then we can get that 
	\[
		\Lip(\xi\vert_{\cball(x,r)})\leq r_0/(r_0-r)\leq 1+ 2r/r_0.
	\]
	We put $\widetilde{\varphi}=\varphi\circ \xi:\Unp(\Omega)\to \Omega$. Then
	$\widetilde{\varphi}\vert_{\Omega}=\varphi_{\Omega}$ and
	$\widetilde{\varphi}$ is Lipschitz in a neighborhood of $\overline{W}$.

	Since $E$ is $d$-rectifiable, we see that $\Theta^d(E,x)=1$ and 
	$T_x=\Tan^d(E,x)$ exists for $\HM^d$-a.e. $x\in E$. By Lemma \ref{le:ld}, we
	get that for $\HM^d$-a.e. $x\in E$, \eqref{eq:ld} holds. So there is a 
	$r_{\tau}(x)>0$ such that for any $0<r\leq r_{\tau}(x)$,
	\[
		\begin{gathered}
			\left| \frac{\HM^d(E\cap \cball(x,r))}{\omega_d r^d} -1\right|\leq \tau,\
			\frac{\HM^d\left(E\cap \cball(x,r)\setminus
			\mathcal{C}(T_x,x,\tau)\right)}{\omega_d r^d}\leq \tau,\\
			\left|\frac{\HM^d(\widetilde{\varphi}(T_x\cap \cball(x,r)))}{\omega_d r^d}-
			\ap J_d(\widetilde{\varphi}\vert_E)(x)\right|\leq  \tau
		\end{gathered}
	\]
	and
	\[
		\left|\frac{\HM^d(\widetilde{\varphi}(E\cap \cball(x,r)))}{\omega_d r^d}-
		\ap J_d(\widetilde{\varphi}\vert_E)(x)\right|\leq  \tau.
	\]
	Since $\widetilde{\varphi}\vert_{\Omega}=\varphi_{\Omega}$, we get that
	\[
		\HM^d(\widetilde{\varphi}(T_x\cap \cball(x,r)))\leq
		\HM^d(\varphi(E\cap \cball(x,r)))+2 \tau \omega_d r^d.
	\]
	We denote by $E'$ the points $x\in E$ such that $\Theta^d(E,x)=1$,
	$T_x=\Tan^dE,x)$ exists, and \eqref{eq:ld} holds. Then $\HM^d(E\setminus
	E')=0$. We take $\delta=\delta(\tau)>0$ such
	that $\HM^d(E\cap W\setminus W^{3 \delta})< \tau$. Then 
	\[
		\mathscr{C}=\left\{ \cball(x,r):x\in E'\cap W^{3\delta}, 0<r< \min\{r_{\tau}(x),
		\delta,r_0/2\}, \HM^d(E\cap \partial \cball(x,r))=0\right\}
	\]
	is a Vitali covering of $E'\cap W^{3\delta}$. By Vitali covering theorem, we
	can find countable many balls $\{\cball(x_i,r_i)\}_{i\in I}\subseteq
	\mathscr{C}$ such that 
	\[
		\HM^d\left(E'\cap W^{3\delta}\setminus \bigcup_{i\in
		I}\cball(x_i,r_i)\right)=0.
	\]
	So there is a positive integer $N>0$ such that
	\[
		\HM^d\left(E'\cap W^{3\delta}\setminus
		\bigcup_{i=1}^N\cball(x_i,r_i)\right)< \tau.
	\]

	We put $\cball_i=\cball(x_i,r_i)$, $\oball_i=\oball(x_i,r_i)$, $T_i=\Tan^dE,x_i)$ and
	$A_i=(T_i+\cball(0,\tau r_i) )\cap \cball(x_i,(1-\tau)r_i)$.
	By Lemma \ref{le:mp}, there exit Lipschitz mappings $g_i:\mathbb{R}^n\to
	\mathbb{R}^n$ such that 
	\[
		g_i\vert_{\mathbb{R}^n\setminus \oball_i}=\id_{\mathbb{R}^n\setminus U_i},
		\ g_i\vert_{A_i}=(T_i)_{\natural}\vert_{A_i},\ g_i(\oball_i)\subseteq
		\oball_i \text{ and } \Lip(g_i)\leq 4.
	\]
	Since $\Lip(\xi\vert_{\cball_i})\leq 1+2r/r_0\leq 2$, we get that $\xi\circ
	g_i(\oball_i)\subseteq \Omega\cap 2 \oball_i\subseteq W^{\delta}$.
	We put $g=g_N\circ\cdots\circ g_1$ and $U'=\cup_{1\leq i\leq N}\oball_i$. 
	Since $\cball_i$, $1\leq i\leq N$, are mutual disjoint, we get that 
	\[
		g\vert_{\mathbb{R}^n\setminus U'}=\id_{\mathbb{R}^n\setminus
		U'},\ g(U')\subseteq U' \text{ and } \Lip(g)\leq 4.
	\]
	It is easy to see that $\widetilde{\varphi}\circ g=\varphi\circ\xi\circ
	g\in \mathcal{D}(\mathbb{R}^n,U)$ and $W_{\widetilde{\varphi}\circ g}
	\subseteq W$. And
	\[
		\begin{aligned}
			\HM^d(E\cap \cball_i\setminus A_i)&\leq \HM^d(E\cap \cball_i\setminus
			\cball(x_i,(1- \tau)r_i))+\HM^d(E\cap \cball_i\setminus
			\mathcal{C}(T_i,x_i,\tau))\\
			&\leq (1+\tau)\omega_d r_i^d-(1-\tau)\omega_d (1-
			\tau)^dr_i^d + \tau \omega_d r_i^d\\
			&\leq (d+2)\omega_d \tau r_i^d.
		\end{aligned}
	\]
	Put $r=\diam(W\cap \varphi(W))$. By Corollary \ref{co:lusc}, we get that 
	\[
		\begin{aligned}
			\HM^d(E\cap W)&\leq \liminf_{m\to \infty}\HM^d(E_m\cap W)\leq
			\liminf_{m\to \infty}(M_m(r)\HM^d(\widetilde{\varphi}\circ g(E_m\cap
			W))+\varepsilon_m(r))\\
			&\leq M(r)\liminf_{m\to \infty}\HM^d(\widetilde{\varphi}\circ g(E_m\cap
			W))+\varepsilon(r)\\
			&\leq M(r)\liminf_{m\to \infty}\left(\sum_{i=1}^{N}\HM^d(\widetilde{\varphi}
			\circ g(E_m\cap
			\cball_i))+\HM^d\Big(\widetilde{\varphi}\circ g\big(E_m\cap W\setminus
			\cup_{i=1}^{N}\cball_i\big)\Big)\right)+\varepsilon(r)\\
			&\leq M(r)\liminf_{m\to \infty}\sum_{i=1}^{N}\HM^d(\widetilde{\varphi}\circ 
			g(E_m\cap A_i))+\varepsilon(r)\\
			&\quad + M(r)(8\Lip(\varphi))^d\limsup_{m\to \infty}\left(\sum_{i=1}^{N}
			\HM^d(E_m\cap \cball_i \setminus A_i)+\HM^d\Big(E_m\cap W\setminus 
			\cup_{i=1}^{N}\cball_i\Big)\right)\\
			&\leq M(r)\sum_{i=1}^{N}\HM^d(\widetilde{\varphi}(T_i\cap
			\cball_i))+M(r)(8\Lip(\varphi))^d\left(2 \tau+\sum_{i=1}^N (d+2)
			\omega_d \tau r_i^d\right)+\varepsilon(r)\\
			&\leq M(r)\sum_{i=1}^{N}\HM^d(\varphi(E\cap
			\cball_i))+M(r)(8\Lip(\varphi))^d\left(2 \tau+\sum_{i=1}^N (d+4)
			\omega_d \tau r_i^d\right)+\varepsilon(r)\\
			&\leq M(r)\HM^d(\varphi(E\cap W))+M(r)(8\Lip(\varphi))^d\left(2 
			\tau+\sum_{i=1}^N\frac{ (d+4)
			\tau}{1-\tau}\HM^d(E\cap \cball_i) \right)+\varepsilon(r)\\
			&\leq M(r)\HM^d(\varphi(E\cap
			W))+M(r)(8\Lip(\varphi))^d(2+(2d+8)\HM^d(E\cap W))\tau+\varepsilon(r).
		\end{aligned}
	\]
	Let $\tau\to 0$, we get that
	\[
		\HM^d(E\cap W)\leq M(r)\HM^d(\varphi(E\cap W))+\varepsilon(r),
	\]
	thus $E\in \QM_d(\Omega,U,M,\varepsilon)$.

	By Proposition \ref{prop:limitradon} and Theorem \ref{thm:vcp}, we get that
	$\mu\mr U$ is $d$-rectifiable and $1\leq \Theta^d(\mu,x)\leq M(0+)$ for
	$\HM^d$-a.e. $x\in E$. Thus for any open set $O\subseteq U$ and compact set
	$H\subseteq U$, we have that 
	\[
		\HM^d(E\cap O)\leq \mu(O)\leq \liminf_{k\to \infty}(\HM^d\mr E_k) (O) =
		\liminf_{k\to \infty}\HM^d(E_k\cap O)
	\]
	and 
	\[
		M(0+)\HM^d(E\cap H)\geq \mu(H)\geq \limsup_{k\to \infty}(\HM^d\mr E_k) (H)
		=\limsup_{k\to \infty}\HM^d( E_k\cap H).
	\]
\end{proof}

\begin{proof}[Proof of Theorem \ref{thm:qmv}]
	Since $\mu\mr U$ is $d$-rectifiable and $\Theta^d(\mu,x)=1$ for $\HM^d$-a.e.
	$x\in E$, we get from Theorem \ref{thm:vcp} that $V\mr U\times \grass{n}{d}
	=\var(E)$.
\end{proof}
\begin{proof}[Proof of Corollary \ref{co:convvarpre}] 
	Without loss of generality, we assume that $\HM^d(E_k)<\infty$. Put 
	\[
		\zeta_k=\sup\big\{\HM^d(E_k)-\HM^d(\varphi(E_k)):\varphi\in
		\mathcal{D}(\Omega,U)\big\}.
	\]
	Then we get that $\zeta_k=\mathscr{E}(E_k)\geq 0$ and $\zeta_k\to 0$ as
	$k\to \infty$. Put  $M_k\equiv 1$ and $\varepsilon_k\equiv \zeta_k$,
	$M(r)=\limsup_{k\to \infty}M_k(r)$ and $\varepsilon(r)=\limsup_{k\to
	\infty}\varepsilon_k(r)$. Then we see that $E_k\in
	\QM_d(\Omega,U,M_k,\varepsilon_k)$, $M\equiv 1$ and $\varepsilon
	\equiv 0$. By Theorem \ref{thm:vcp}, setting $E=U\cap \spt\|V\|$, we have
	that $V\mr U\times \grass{n}{d}=\var(E)$, and by Theorem \ref{thm:QM}, we
	get that $E$ is minimal in $U\cap \Omega$.
\end{proof}

\section{Convergence of almost minimal sets}
\begin{lemma}
	Suppose that $E_k\in \QM_d(\mathbb{R}^n,U,M_k,\varepsilon_k)$, $M(r)
	=\limsup_{k}M_k(r)$ and $\varepsilon(r)=\limsup_k \varepsilon_k(r)$, 
	$\HM^d\mr (E_k\cap U) \wc \mu$. If $M(0+)<\infty$ and there exists $r_0>0$ 
	such that 
	\[
		\int_{0}^{r_0}r^{-d-1}\varepsilon(r)\ud r <\infty,
	\]
	then, by setting $\varepsilon_1(r)=\int_{0}^{r}t^{-d-1}\varepsilon(t)\ud t$, 
	we have that for $\HM^d$-a.e. $x\in  U\cap \spt \mu$
	and any $0<r<\min\{r_0, \dist (x,\mathbb{R}^n\setminus U)\}$,
	\[
		\mu(\cball(x,r))\geq \left(c_0^{-d-1}(3
		\sqrt{n})^{d-n}d^{-d}M(r)^{-d}-\frac{c_0M(r)d}{c_0M(r)-1}\Big(\varepsilon(r)+\varepsilon_1(r)\Big)\right)r^d.
	\]
\end{lemma}
\begin{proof}
	Since $\int_{0}^{r_0}r^{-d-1}\varepsilon(r)\ud r <\infty$, we get that 
	$\limsup_{r\to 0+}r^{-d}\varepsilon(r)=0$. By 
	Lemma \ref{le:lden}, we see that for $\mu$-a.e. $x\in U\cap \spt\mu$,
	$\Theta_{\ast}^d(\mu,x)\geq (2 \omega_d c_0 M(0+))^{-1}$. We take a such 
	point $x$ and $0<r<\min\{r_0, \dist (x,\mathbb{R}^n\setminus U)$.
	For any $0<t\leq r$, we put $v_k(t)=\HM^d(E_k\cap \cball(x,t))$,
	$v(t)=\mu(\cball(x,t))$, and 
	\[
		\alpha=\left(\frac{c_0 M(r)-1}{c_0M(r)}\right)^{1/d}.
	\]
	Then $1/4<\alpha\leq 1-1/(dc_0M(r))$. If 
	$ v(r)<c_0^{-d-1}(3\sqrt{n})^{d-n}d^{-d}M(r)^{-d}r^d$,
	then
	\[
		\limsup_{k\to \infty}\HM^d(E_k\cap \cball(x,r))\leq v(r)< c_0^{-1}
		(3\sqrt{n})^{d-n}(r-\alpha r)^d.
	\]
	We take positive integer $m$ such that 
	\[
		\limsup_{k\to \infty}\HM^d(E_k\cap \cball(x,\alpha^k r))< c_0^{-1}
		(3\sqrt{n})^{d-n}(\alpha^{\ell} r-\alpha^{\ell+1} r)^d, \text{ for }
		0\leq \ell\leq m-1,
	\]
	and 
	\[
		\limsup_{k\to \infty}\HM^d(E_k\cap \cball(x,\alpha^m r))\geq c_0^{-1}
		(3\sqrt{n})^{d-n}(\alpha^m r-\alpha^{m+1} r)^d.
	\]

	For $k$ large enough, $\HM^d(E_k\cap \cball(x,r))< c_0^{-1}
	(3\sqrt{n})^{d-n}(r-\alpha r)^d$.
	By Lemma \ref{le:ffpa}, there is a Lipschitz mapping $\phi_k:\mathbb{R}^n\to
	\mathbb{R}^n$ such that $\phi_k(\oball(x,r))\subseteq \oball(x,r)$, 
	$ \phi_k\vert_{\mathbb{R}^n\setminus
	\oball(x,r)}=\id_{\mathbb{R}^n\setminus \oball(x,r)}$ and
	\[
		\HM^d(\phi_k(E_k\cap \oball(x,r)))\leq c_0\HM^d(E_k\cap \ann(x,\alpha r,r)).
	\]
	Since $E_k\in \QM_d(\mathbb{R}^n,U,M_k,\varepsilon_k)$, we get that 
	\[
		\HM^d(E_k\cap \oball(x,r))\leq M_k(r) \HM^d(\phi(E_k\cap \oball(x,r))) +
		\varepsilon_k(r),
	\]
	thus
	\[
		v_k(r)\leq c_0M_k(r)(v_k(r)-v_k(\alpha r))+ \varepsilon_k(r).
	\]
	Similarly, we can get that for $1\leq \ell \leq m-1$,
	\[
		\begin{aligned}
			v_k(\alpha^{\ell} r)&\leq c_0M_k(\alpha^{\ell} r)(v_k(\alpha^{\ell}
			r)-v_k(\alpha^{\ell+1} r))+ \varepsilon_k(\alpha^{\ell} r)\\
			&\leq c_0M_k(r)(v_k(\alpha^{\ell} r)-v_k(\alpha^{\ell+1} r))+
			\varepsilon_k(\alpha^{\ell} r).
		\end{aligned}
	\]
	Hence 
	\[
		v_k(\alpha^{\ell+1} r)\leq \frac{c_0 M_k(r)-1}{c_0
		M_k(r)}v_k(\alpha^{\ell} r) +\frac{1}{c_0
		M_k(r)}\varepsilon_k(\alpha^{\ell} r),
	\]
	and 
	\[
		v_k(\alpha^m r)\leq  \left( \frac{c_0M_k(r)-1}{c_0M_k(r)}\right)^m
		v_k(r)+\frac{1}{c_0M_k(r)}\sum_{\ell=0}^{m-1}\left(\frac{c_0M_k(r)-1}
		{c_0M_k(r)}\right)^{\ell} \varepsilon_k(\alpha^{m-1-\ell}r).
	\]
	Thus
	\[
		v_k(r)\geq \left( \frac{c_0M_k(r)}{c_0M_k(r)-1}\right)^m v_k(\alpha^m r)-
		\frac{1}{c_0M_k(r)-1}\sum_{\ell=0}^{m-1}\left(\frac{c_0M_k(r)-1}{c_0M_k(r)}
		\right)^{\ell} \varepsilon_k(\alpha^{\ell}r)
	\]
	and 
	\[
		\begin{aligned}
			v(r)&\geq \limsup_{k\to \infty}v_k(r)\\
			&\geq \left(
			\frac{c_0M(r)}{c_0M(r)-1}\right)^m \limsup_{k\to \infty}v_k(\alpha^m r)-
			\frac{1}{c_0M(r)-1}\sum_{\ell=0}^{m-1}\left(\frac{c_0M(r)-1}{c_0M(r)}
			\right)^{\ell} \varepsilon(\alpha^{\ell}r) \\
			&= \alpha^{-md}c_0^{-1} (3\sqrt{n})^{d-n}(\alpha^m r-\alpha^{m+1} r)^d-
			\frac{1}{c_0M(r)-1}\sum_{\ell=0}^{m-1} \alpha^{-\ell d}
			\varepsilon(\alpha^{\ell} r).\\
		\end{aligned}
	\]
	Since 
	\[
		\sum_{\ell=1}^{m-1}\alpha^{\ell d}\varepsilon(\alpha^{\ell
		d}r)\leq\sum_{\ell=1}^{m-1}\frac{r^d}{\ln 1/\alpha} \int_{\alpha^{\ell}
		r}^{\alpha^{\ell-1}r}t^{-d-1}\varepsilon(t) \ud t \leq \frac{r^d}{\ln
		1/\alpha} \varepsilon_1(r)
	\]
	and 
	\[
		\ln 1/\alpha=-\frac{1}{d}\ln \left(1-\frac{1}{c_0M(r)}\right)\geq
		\frac{1}{dc_0M(r)},
	\]
	we get that
	\[
		\begin{aligned}
			v(r)&\geq c_0^{-1}(3 \sqrt{n})^{d-n}(1-\alpha)^d r^d
			-\frac{dc_0M(r)}{c_0M(r)-1}\left(\varepsilon(r)+\varepsilon_1(r)\right)r^d\\
			&\geq \left(c_0^{-d-1}(3 \sqrt{n})^{d-n}d^{-d}M(r)^{-d}-
			\frac{dc_0M(r)}{c_0M(r)-1}\Big(\varepsilon(r)+\varepsilon_1(r)\Big)\right)r^d.
		\end{aligned}
	\]
\end{proof}
\begin{corollary}\label{co:AR}
	Suppose that $E\in \QM_d(\mathbb{R}^n,U,M,\varepsilon)$.
	If $ \int_{0}^{r_0}r^{-d-1}\varepsilon(r)\ud r <\infty$ for some $r_0>0$,
	then there is a constant $C=C(n,d,M(0+))>0$ and $0<r_1\leq r_0$ such that 
	for any $x\in E\cap U$ and $0<r\leq
	\min\{r_1,\dist(x,\mathbb{R}^n\setminus U)\}$,
	\begin{equation}\label{eq:AR}
		C^{-1}r^{d}\leq \HM^d(E\cap \cball(x,r))\leq C r^d.
	\end{equation}
\end{corollary}
\begin{proof}
	By Lemma \ref{le:uAR} and above lemma, we get that \eqref{eq:AR} holds for
	$\HM^d$-a.e. $x\in E$. By the definition of quasiminimal sets we see that
	$\HM^d(E\cap \oball(x,r))>0$ for any $x\in E\cap U$ and $r>0$, thus
	\eqref{eq:AR} holds for every $x\in E$.
\end{proof}
\begin{lemma}
	\label{le:HDV}
	Suppose that $\{E_k\}\subseteq \QM_d(U,U,M,\varepsilon)$ and $
	\int_{0}^{r_0}r^{-d-1}\varepsilon(r)\ud r <\infty$ for some $r_0>0$. If
	$\HM^d\mr E_k$ converges in weak topology to a Radon measure $\mu$  
	with $E= U\cap \spt \mu$, then $E_k$ converges to $E$ in $U$ in local 
	Hausdorff distance; conversely, if $E_k$ converges to $E$ in $U$ in local 
	Hausdorff distance and $\HM^d\mr E_k\wc \mu$, then $U\cap \spt\mu = E$.
\end{lemma}
\begin{proof}
	If $\HM^d\mr E_k\wc \mu$, $E=U\cap \spt\mu$, then by Theorem \ref{thm:QM}, we get
	that $E\in \QM_d(U,U,M,\varepsilon)$. Let $\{E_{k_i}\}_{i=1}^{\infty}$ be a
	subsequence which converges to a set $F$ in $U$ in local Hausdorff distance.
	Then for any $x\in E$ and $r>0$ with $\oball(x,r)\subseteq U$, we have 
	\[
		0<\mu(\oball(x,r))\leq \liminf_{i\to \infty} \HM^d(E_{k_i}\cap \oball(x,
		r)),
	\]
	thus $E_{k_i}\cap \oball(x,r)\neq \emptyset$ for $i$ large enough, hence
	\[
		\dist(x,F)\leq \limsup_{i\to \infty}\dist(x,E_{k_i})\leq r.
	\]
	By the arbitrariness of $r>0$, we get that $E\subseteq F$. Since $F$ is the
	local Hausdorff distance limit of $\{E_{k_i}\}$, we get that for any $y\in F$,
	$ \dist(y,E_{k_i})\to 0$. So there exist $z_i\in E_{k_i}$ such that $|y-z_i|\to
	0$. By the corollary above, there exist $C>0$ and $r_1>0$ such that
	for any $0<r<r_1$, 
	\[
		C^{-1}r^{d}\leq \HM^d(E_{k_i}\cap \cball(x,r))\leq Cr^d,
	\]
	when $\oball(y, 2r)\subseteq U$ and $i$ is large so that $|y-z_i|<r$. Thus 
	\[
		\mu(\cball(y,r))\geq \limsup_{i\to \infty}\HM^d(E_{k_i}\cap \cball(y,r))
		\geq C^{-1} r^{d}.
	\]
	Hence $y\in U\cap\spt\mu = E$. So $E = F$. We get that $\{E_k\}$ converges to
	$E$ in $U$ in local Hausdorff distance.

	If $E_k$ converges to $E$ in $U$ in local Hausdorff distance and $\HM^d\mr 
	E_{k}$ converges to a Radon measure $\mu$. Then for the same reason we will 
	get both $U\cap \spt\mu \subseteq E$ and $E\subseteq U\cap \spt \mu$.
\end{proof}

\begin{corollary}
	Suppose $1\leq d<m\leq n$. Let $\Omega\subseteq \mathbb{R}^n$ be a closed
	$m$-dimensional submanifold  of class $C^2$. Let $U\subseteq \mathbb{R}^n$
	be an open set such that $U\cap \Omega\neq \emptyset$ and $U\cap \partial
	\Omega=\emptyset$. Let $\{E_k\}\subseteq \QM_d(\Omega\cap
	U,U,M,\varepsilon)$ be a sequence of quasiminimal sets. Suppose that $E_k$
	converges to $E$ in $U$ in local Hausdorff distance. If $
	\int_{0}^{r_0}r^{-d-1}\varepsilon(r)\ud r
	<\infty$ for some $r_0>0$, then $E\in \QM_d(\Omega\cap U,U,M,\varepsilon)$,
	and if in addition $\HM^d(E)<\infty$ and $\HM^d(E_k)\to \HM^d(E)$, we will
	get that  $ \var(E_k) \wc \var(E) $.
\end{corollary}
\begin{proof}
	Let $\{E_{k_i}\}$ be a subsequence such that $\HM^d\mr E_{k_i}$ converges to
	a Radon $\mu$. Then by the lemma above, we get that $E=U\cap \spt\mu$. By
	Theorem \ref{thm:QM}, we get that $E\in \QM_d(\Omega\cap U,U,M,\varepsilon)$.
	For any $x\in E$, we put $r_x=\dist(x,\mathbb{R}^n\setminus U)$. Since
	$\HM^d(E)<\infty$, we get that for $H^1$-a.e. $r\in (0,r_x)$, $\HM^d(E\cap
	\partial \cball(x,r)) =0$. Then for such $r$, 
	\[
		\HM^d(E\cap \oball(x,r))\leq \liminf_{k\to \infty}\HM^d(E_k\cap
	\oball(x,r)) \]
	and 
	\[
		\HM^d(E\cap (U\setminus \cball(x,r)))\leq \liminf_{k\to \infty}\HM^d(E_k
		\cap (U\setminus \cball(x,r))).
	\]
	Thus 
	\[
		\begin{aligned}
			\HM^d(E\cap \cball(x,r))& = \HM^d(E)-\HM^d(E\cap (U\setminus \cball(x,
			r)))\\
			&\geq \lim_{k\to\infty}\HM^d(E)- \liminf_{k\to \infty}\HM^d(E_k
			\cap (U\setminus \cball(x,r)))\\
			&\geq \limsup_{k\to \infty}\HM^d(E_k\cap \cball(x,r))\geq \limsup_{k\to
			\infty}\HM^d(E_k\cap \oball(x,r))\\
			&\geq \liminf_{k\to \infty}\HM^d(E_k\cap \oball(x,r))\geq \HM^d(E\cap
			\oball(x,r)).
		\end{aligned}
	\]
	But $\HM^d(E\cap \partial \cball(x,r))=0$, we get that 
	\[
		\lim_{k\to \infty}\HM^d(E_k\cap \oball(x,r))= \HM^d(E\cap \oball(x,r)).
	\]
	By Theorem \ref{thm:vcp}, we get that $\var(E_k)\wc \var(E)$.
\end{proof}

\begin{proof}[Proof of Theorem \ref{thm:amostminimal}]
	Let $\{E_k\}\subseteq\AM_d(\Omega\cap U,M,\varepsilon)$ be a sequence such
	that $E_k$ converges to $E$ in $U$ in local Hausdorff distance. Then by the
	corollary above, we get that $E\in \AM_d(\Omega\cap U,M,\varepsilon)$. Thus
	$\AM_d(\Omega\cap U,M,\varepsilon)$ is compact with respect to the local
	Hausdorff distance. Since $M(0+)=1$, by Theorem \ref{thm:vcp}, we have that 
	$\var(E_k)\wc \var(E)$. And conversely, if $\{E_k\}\subseteq \AM_d(\Omega\cap
	U,M,\varepsilon)$ is a sequence of almost minimal sets such that $\var(E_k)
	\wc V$, then $E=U\cap \spt\|V\|\in \AM_d(\Omega\cap U,M,\varepsilon)$. By Lemma
	\ref{le:HDV}, we get that $E_k$ converges to $E$ in $U$ in local Hausdorff
	distance. Thus the mapping $\var:\AM_d(\Omega\cap U,M,\varepsilon)\to
	\{\var(E): E\in \AM_d(\Omega\cap U,M,\varepsilon)\}$ is a homeomorphism.
\end{proof}

\bibliography{gmt}
\end{document}